\pdfoutput=1
\documentclass{article}
\usepackage[english]{babel}
\usepackage[a4paper,hscale=.7,bottom=35mm,top=28mm]{geometry}
\usepackage{amsmath,amssymb}
\usepackage{amsthm,thmtools}
\usepackage[shortcuts]{extdash}
\usepackage{enumitem}
\usepackage{lmodern,microtype}
\usepackage{tikz,tikz-cd}
\usepackage{hyperref}

\author{Josse van Dobben de Bruyn}
\title{The Archimedean order unitization of seminormed~ordered~vector~spaces}
\date{13 February 2025}

\makeatletter
\newcommand{\myaddrA}[1]{\gdef\my@addressA{\par\textsc{#1}}}
\newcommand{\myaddrB}[1]{\gdef\my@addressB{\par\textsc{#1}}}
\newcommand{\mycuraddr}[1]{\gdef\my@curaddr{\par\textit{Current address:} #1.}}
\newcommand{\myemail}[1]{\gdef\my@email{\par\textit{E-mail address:} \texttt{\href{mailto:#1}{#1}.}}}
\AtEndDocument{\par\bigskip\my@addressA\my@addressB\my@curaddr\my@email}
\newcommand{\mysubjclass}[2][2010]{\gdef\my@subjclass{#1 \textit{Mathematics Subject Classification}. #2.}}
\newcommand{\mykeywords}[1]{\gdef\my@keywords{\textit{Key words and phrases}. #1.}}
\newcommand{\mymaketitle}{%
	\let\@oldthanks\@thanks%
	\gdef\@thanks{\@oldthanks\footnotetext{\my@subjclass}\footnotetext{\my@keywords}}%
	\maketitle
}
\makeatother

\myaddrA{Mathematical Institute, Leiden University, P.O.~Box~9512, 2300~RA~Leiden, The~Netherlands}
\myaddrB{Delft Institute of Applied Mathematics, Delft University of Technology, Mekelweg~4, 2628~CD~Delft, The~Netherlands}
\mycuraddr{Department of Applied Mathematics and Computer Science, Technical University of Denmark, Richard Petersens Plads, Building~324, 2800~Kongens~Lyngby, Denmark}
\myemail{jdob@dtu.dk}
\mysubjclass[2020]{Primary: 46B40, Secondary: 06F20, 46A40, 47B65, 47L25, 46L07}
\mykeywords{Seminormed preordered vector space, convex cone, Archimedean order unit space, adjoining an order unit, unitization, full hull, normal cone, matrix ordered operator space, non-unital operator system}

\makeatletter
\def\mylinkcolor{blue!55!black}
\hypersetup{pdfauthor=\@author,pdftitle=\@title,colorlinks=true,linkcolor=\mylinkcolor,citecolor=\mylinkcolor,urlcolor=\mylinkcolor}
\makeatother

\newcommand*{\mysecref}[1]{\hyperref[#1]{\S\ref*{#1}}}
\newcommand*{\myautoref}[2]{\hyperref[#2]{\autoref*{#1}\ref*{#2}}}  
\newcommand{\myref}[1]{\textit{\ref{#1}}}
\newcommand{\myeqautoref}[1]{\hyperref[#1]{Equation~\eqref{#1}}}

\newcommand{\hair}{\ifmmode\mskip1mu\else\kern0.08em\fi}
\newcommand{\exampleqed}{\ensuremath{\scriptstyle\triangle}}

\numberwithin{equation}{section}

\declaretheorem[style=definition,sibling=equation]{definition}
\declaretheorem[style=definition,sibling=definition]{construction}

\declaretheorem[style=remark,qed=\exampleqed,sibling=definition]{example}
\declaretheorem[style=remark,sibling=definition]{remark}

\declaretheorem[style=plain,sibling=definition]{theorem}
\declaretheorem[style=plain,sibling=definition]{lemma}
\declaretheorem[style=plain,sibling=definition]{proposition}
\declaretheorem[style=plain,sibling=definition]{corollary}
\declaretheorem[style=plain,sibling=definition]{claim}

\newcommand{\N}{\ensuremath{\mathbb{N}}}

\newcommand{\R}{\ensuremath{\mathbb{R}}}
\providecommand{\C}{}
\renewcommand{\C}{\ensuremath{\mathbb{C}}}

\DeclareSymbolFont{bbold}{U}{bbold}{m}{n}
\DeclareSymbolFontAlphabet{\mathbbold}{bbold}
\newcommand{\one}{\ensuremath{\mathbbold{1}}}

\newcommand{\mywedge}{\mathcal K}

\newcommand{\topdual}{'}
\newcommand{\algdual}{^*}
\newcommand{\weakstar}{{w*}}

\newcommand{\lineal}{\operatorname{lin.\!space}}
\DeclareMathOperator{\fh}{fh}
\DeclareMathOperator{\spn}{span}
\DeclareMathOperator{\ex}{ex}
\newcommand{\sa}{{\mkern2mu\operatorname{sa}}}
\let\sc\relax
\newcommand{\sc}{{\mkern2mu\operatorname{sc}}}
\DeclareMathOperator{\id}{id}
\DeclareMathOperator{\im}{im}

\newcommand{\lintop}{\mathfrak T}

\newcommand{\unitArch}[1]{\ensuremath{#1_{\operatorname{Arch},1}}}
\newcommand{\Arch}[1]{\ensuremath{#1_{\operatorname{Arch}}}}
\newcommand{\unit}[1]{\ensuremath{#1_{1}}}

\newcommand{\mUnitArch}[1]{\ensuremath{\tilde{#1}}}
\newcommand{\mArch}[1]{\ensuremath{\tilde{#1}}}
\newcommand{\mUnit}[1]{\ensuremath{#1_{1}}}

\newcommand{\pu}[1]{\ensuremath{#1^\sharp}}

\newcommand{\cosegment}[3][]{\ell_{#1[#2,#3#1)}}

\def\Hom{L}

\begin{document}
\mymaketitle
\begin{abstract}
	In this paper, we describe a way of turning a seminormed preordered vector space into an Archimedean order unit space.
	We show that this construction satisfies a universal property similar to that of the Archimedeanization of Paulsen and Tomforde, and we give a number of applications of our result in ordered vector spaces and in matrix ordered operator spaces.
	In ordered vector spaces, we use our our Archimedean order unitzation to shed new light on normality criteria for seminorms.
	In matrix ordered operator spaces, we prove several new results about Werner's ``partial unitization'': we give a simplified ``internal'' description of the positive cone of Werner's partial unitization, and we prove a necessary and sufficient condition for the embedding of a matrix ordered operator space in its partial unitization to be a complete isomorphism.
	This last result was already announced in Werner's 2002 paper, but to our knowledge no proof exists in the literature.
\end{abstract}

\section{Introduction}
Among the most well-behaved ordered vector spaces are the Archimedean order unit (AOU) spaces.
When equipped with the order unit norm, these spaces have many favourable properties, including a famous representation theorem by Kadison \cite[Theorem 2.1]{Kadison-representation} and, in many cases, automatic continuity of positive linear operators to or from these spaces (see \cite[Proposition~II.2.16]{Peressini}).
It is therefore useful to have ways of turning other types of spaces into an AOU space.

A non-Archimedean order unit space can be turned into an AOU space by using the \emph{Archimedeanization} of Paulsen and Tomforde \cite{Paulsen-Tomforde}, given by the following universal property.
\begin{theorem}[{Paulsen and Tomforde, \cite{Paulsen-Tomforde}}]
	\label{thm:intro:Paulsen-Tomforde}
	For every order unit space $E$, there is an AOU space $\unitArch{E}$ and a unital positive linear map $\phi : E \to \unitArch{E}$ with the following universal property:
	for every AOU space $F$ and every unital positive linear map $\psi : E \to F$, there is a unique unital positive linear map $\mUnitArch{\psi} : \unitArch{E} \to F$ such that $\psi = \mUnitArch{\psi} \circ \phi$.
	\[ \begin{tikzcd}[column sep={2.5em,between origins}]
		& E\arrow[dl,"\phi",swap]\arrow[dr,"\psi"] \\
		\unitArch{E}\arrow[rr,"\mUnitArch{\psi}",densely dashed] & & F
	\end{tikzcd} \]
	By the universal property, $\unitArch{E}$ is unique up to a unique unital order isomorphism.
\end{theorem}

The Archimedeanization $\unitArch{E}$ can be constructed by taking the closure of $E_+$ with respect to the order unit seminorm and dividing out the lineality space of the resulting cone.
This construction turns an order unit space into an AOU space in a canonical way, but it does not apply to spaces without order units.

\pagebreak
Emelyanov \cite{Emelyanov} extended the Archimedeanization to spaces without an order unit.

\begin{theorem}[{Emelyanov, \cite{Emelyanov}}]
	\label{thm:intro:Emelyanov}
	For every ordered vector space $E$, there is an Archimedean ordered vector space $\Arch{E}$ and a positive linear map $\phi : E \to \Arch{E}$ with the following universal property:
	for every Archimedean ordered vector space $F$ and every positive linear map $\psi : E \to F$, there is a unique positive linear map $\mArch{\psi} : \Arch{E} \to F$ such that $\psi = \mArch{\psi} \circ \phi$.
	\[ \begin{tikzcd}[column sep={2.5em,between origins}]
		& E\arrow[dl,"\phi",swap]\arrow[dr,"\psi"] \\
		\Arch{E}\arrow[rr,"\mArch{\psi}",densely dashed] & & F
	\end{tikzcd} \]
	By the universal property, $\Arch{E}$ is unique up to a unique order isomorphism.
\end{theorem}

Emelyanov's construction is equivalent to taking the algebraically closed hull (in the sense of \cite[\S 16.2.(5)]{Kothe-I}) of the positive cone $E_+$ and then dividing out the lineality space of the resulting cone.
However, this construction does not add an order unit, so the Archimedeanization of an arbitrary space is generally not an AOU space.

The aforementioned results can be used to make a space Archimedean, but they do not add an order unit.
In this paper, we give a related construction that simultaneously makes the space Archimedean and adds an order unit.
To do so, we require a little bit of extra structure: we assume that the space is equipped with a seminorm.

If $E$ is a seminormed preordered vector space and $F$ is an AOU space, we say that a linear map $\psi : E \to F$ is \emph{contractive} if $\lVert \psi(x) \rVert_u \leq p(x)$ for all $x \in E$, where $p$ is the seminorm of $E$ and $\lVert \:\cdot\: \rVert_u$ is the order unit norm on $F$.
We prove the following.

\begin{theorem}
	\label{thm:intro:universal-property}
	For every seminormed preordered vector space $E$, there is an AOU space $\unit{E}$ and a contractive positive linear map $\phi : E \to \unit{E}$ with the following universal property:
	for every AOU space $F$ and every contractive positive linear map $\psi : E \to F$, there is a unique unital positive linear map $\mUnit{\psi} : \unit{E} \to F$ such that $\psi = \mUnit{\psi} \circ \phi$.
	\[ \begin{tikzcd}[column sep={2.5em,between origins}]
		& E\arrow[dl,"\phi",swap]\arrow[dr,"\psi"] \\
		\unit{E}\arrow[rr,"\mUnit{\psi}",densely dashed] & & F
	\end{tikzcd} \]
	By the universal property, $\unit{E}$ is unique up to a unique unital order isomorphism.
\end{theorem}

There are several ways to construct $\unit{E}$.
We take a geometric point of view, based on the following ``internal'' description of $\unit{E}$.
Letting $N = \overline{E_+} \cap -\overline{E_+}$ denote the lineality space of the closure of $E_+$, we construct the Archimedean order unitization as the space $\unit{E} = (E/N) \oplus \R$ with the positive cone
\begin{equation}
	(\unit{E})_+ := \{ (x + N , \lambda) \, : \, d(x,E_+) \leq \lambda \} \label{eqn:intro:internal-definition}
\end{equation}
and order unit $(0 + N , 1)$, where $d(x,E_+) = \inf_{y \in E_+} p(x - y)$ denotes the distance between $x$ and $E_+$ with respect to the seminorm of $E$.
This construction will be worked out in more detail in \mysecref{sec:construction}, and we will prove in \mysecref{sec:universal-property} that $\unit{E}$ satisfies the universal property of \autoref{thm:intro:universal-property},

Other constructions are possible as well, and we show along the way that $\unit{E}$ is also given by the following two characterizations.
First, we show that the cone defined in \eqref{eqn:intro:internal-definition} is equal to
\begin{equation}
	(\unit{E})_+ = \big\{ (x + N , \lambda) \, : \, \varphi(x) \geq -\lambda \ \text{for all $\varphi \in E_+\topdual$ with $\lVert \varphi \rVert \leq 1$} \big\}. \label{eqn:intro:external-definition}
\end{equation}
This characterization is analogous to Werner's construction of the ``partial unitization'' of a ``matrix ordered operator space'' \cite[Definition 4.7]{Werner}.
We prove in \autoref{cor:alternative-formulas} that the sets described in \eqref{eqn:intro:internal-definition} and \eqref{eqn:intro:external-definition} are equal.

Second, we show that $\unit{E}$ can also be obtained as follows.
Let $\Omega = \{ \varphi \in E_+\topdual \, : \, \lVert \varphi \rVert \leq 1\}$ be the intersection of the dual cone with the dual unit ball, equipped with the weak\nobreakdash-$*$ topology.
Then $\Omega$ is a compact Hausdorff space and every element $x \in E$ defines a continuous function $\psi_x : \Omega \to \R$, given by $\psi_x(\varphi) = \varphi(x)$.
This gives us a contractive positive linear map $\psi : E \to C(\Omega)$.
In \mysecref{sec:representations}, we show that $\unit{E}$ is isomorphic to the subspace of $C(\Omega)$ spanned by $\im(\psi)$ and $1$.

\pagebreak
We will provide several examples to illustrate the Archimedean order unitization.
In particular, we show that it is compatible with known constructions in the following two cases.
First, if $E$ is an order unit space equipped with the order unit seminorm, then we prove that $\unit{E} \cong \unitArch{E} \oplus \R$, where $\unitArch{E}$ is the Archimedeanization of Paulsen and Tomforde (see \autoref{thm:intro:Paulsen-Tomforde}) and the direct sum is equipped with the pointwise cone $(\unitArch{E})_+ \oplus \R_{\geq 0}$ and order unit $(u + N, 1)$.
Second, if $E = \mathcal A^\sa$ is the self-adjoint part of a $C^*$\nobreakdash-algebra $\mathcal A$, then we prove that $\unit{E} \cong \tilde{\mathcal A}^\sa$ as AOU spaces, where $\tilde A$ is the $C^*$\nobreakdash-algebra unitization of $\mathcal A$.

\subsubsection*{Applications in ordered vector spaces}

Our results provide a number of new insights in ordered vector spaces.
First, it gives us the following ``$1$\nobreakdash-max-normalization'' of a seminorm on an preordered vector space.

Following \cite{Messerschmidt}, we say that a seminorm $p$ on $E$ is \emph{$\alpha$\nobreakdash-max-normal} if $x \leq y \leq z$ implies $p(y) \leq \alpha \max(p(x),p(z))$.
If $p$ is an arbitrary seminorm on a preordered vector space $E$, then the Archimedean order unitization $\phi : E \to \unit{E}$ with respect to $p$ gives us another seminorm $p_u(x) = \lVert \phi(x) \rVert_u$ on $E$.
Since $\phi$ is positive and contractive and $\lVert \:\cdot\: \rVert_u$ is $1$\nobreakdash-max-normal, $p_u$ is a $1$\nobreakdash-max-normal seminorm smaller than $p$.
We call $p_u$ the \emph{$1$\nobreakdash-max-normalization} of $p$, and we study the relation between $p$ and $p_u$.

\begin{theorem}
	\label{thm:intro:seminorm-normalization}
	Let $E$ be a preordered vector space, let $p$ be a seminorm on $E$, and let $p_u$ be the $1$\nobreakdash-max-normalization of $p$.
	Then:
	\begin{enumerate}[label=(\alph*)]
		\item\label{itm:intro:as:formula} $p_u$ is given by $p_u(x) = \max(d(x,E_+),d(-x,E_+)) =  \max \{ |\varphi(x)| \, : \, \varphi \in E_+\topdual,\ \lVert \varphi \rVert = 1 \}$;
		
		\item\label{itm:intro:as:unit-ball} the open unit ball of $p_u$ is the full hull of the open unit ball of $p$;
		
		\item\label{itm:intro:as:normalization} $p_u$ is the largest $1$\nobreakdash-max-normal seminorm smaller than $p$;
		
		\item\label{itm:intro:as:equivalent} for all $\alpha \geq 1$, we have $\frac{1}{\alpha} p \leq p_u \leq p$ if and only if $p$ is $\alpha$\nobreakdash-max-normal;
		
		\item $p$ and $p_u$ are equivalent if and only if the topology induced by $p$ is locally full;
		
		\item\label{itm:intro:as:closure} the closure of $E_+$ with respect to $p$ is equal to the closure of $E_+$ with respect to $p_u$;
		
		\item\label{itm:intro:as:functional} a positive linear functional $\varphi : E \to \R$ is continuous with respect to $p$ if and only if it is continuous with respect to $p_u$, and in this case one has $\lVert \varphi \rVert = \lVert \varphi \rVert_u$.
	\end{enumerate}
\end{theorem}

We note that \ref{itm:intro:as:closure} and \ref{itm:intro:as:functional} hold regardless of whether or not $p$ and $p_u$ are equivalent.
The proof of \autoref{thm:intro:seminorm-normalization} will be given in \mysecref{sec:seminorm-normalization}, except part \ref{itm:intro:as:functional}, which is proved in \autoref{cor:extension-comparison}.

This in turn has an application to normal cones in locally convex spaces.
If $E$ is a locally convex ordered vector space whose positive cone is normal, then it is well-known that the topology of $E$ can be given by a family of monotone seminorms (e.g.{} \cite[Theorem 2.25]{Aliprantis-Tourky}), but it is not so well-known what these seminorms look like.
They are usually constructed by choosing a neighbourhood base consisting of full and absolutely convex sets, and taking the respective Minkowski functionals.
Using \myautoref{thm:intro:seminorm-normalization}{itm:intro:as:formula} and \ref{itm:intro:as:unit-ball}, we get a direct formula for these seminorms.

\begin{theorem}
	\label{thm:intro:lcs-normal}
	Let $E$ be a locally convex ordered vector space whose positive cone is normal.
	If the topology of $E$ is given by the family of seminorms $\{p_\lambda\}_{\lambda \in \Lambda}$, then it is also given by the family of $1$\nobreakdash-max-normal seminorms $\{p_{\lambda,u}\}_{\lambda \in \Lambda}$, where $p_{\lambda,u}(x) = \max(d_\lambda(x,E_+),d_\lambda(-x,E_+))$.
\end{theorem}

Here $d_\lambda(x,E_+) = \inf_{y \in E_+} p_\lambda(x - y)$ denotes the $p_\lambda$\nobreakdash-distance between $x$ and $E_+$.
Since these seminorms are $1$\nobreakdash-max-normal, in particular they are monotone.

For a normed ordered vector space, we have the following special case.

\begin{corollary}
	\label{cor:intro:norm-normal}
	Let $E$ be a normed ordered vector space whose positive cone is normal.
	Then $\lVert x \rVert_u = \max(d(x,E_+),d(-x,E_+))$ defines a $1$\nobreakdash-max-normal norm that is equivalent to the original norm of $E$.
\end{corollary}

These simple formulas are not hard to prove, but they do not seem to be well known.
Moreover, our approach shows that the seminorms $p_{\lambda,u}$ in \autoref{thm:intro:lcs-normal} and the norm $\lVert \:\cdot\: \rVert_u$ in \autoref{cor:intro:norm-normal} can be thought of as coming from the order unit norm in the respective unitization of $E$.

The proofs of \autoref{thm:intro:lcs-normal} and \autoref{cor:intro:norm-normal} will be given in \mysecref{sec:seminorm-normalization}.

\pagebreak

\subsubsection*{Applications to matrix ordered operator spaces}

The Archimedean order unitization from \autoref{thm:intro:universal-property} is similar to Werner's ``partial unitization'' of a ``matrix ordered operator space'' \cite[\S 4]{Werner} (see also \cite{Karn-unitization,Karn-corrigendum,Blecher-et-al}).
Such spaces carry much more data: a matrix ordered operator space consists of a complex $*$\nobreakdash-vector space $E$ along with a norm $\lVert \:\cdot\: \rVert_n$ on $M_n(E)$ for all $n \in \N$ and a convex cone $M_n(E)_+ \subseteq M_n(E)^\sa$ for all $n \in \N$, subject to certain compatibility requirements (see \mysecref{sec:operator-systems} for details).

In \mysecref{sec:operator-systems}, we use the results from this paper to prove new results about Werner's partial unitization.
First, in \mysecref{subsec:pu-simplification}, we use the results from this paper to give a simpler ``internal'' description of the positive cone of Werner's partial unitization, similar to \eqref{eqn:intro:internal-definition}.
The precise statement and proof can be found in \autoref{thm:pu-simplification}.
Second, in \mysecref{subsec:non-unital-operator-systems}, we prove that the embedding $\phi : E \to \pu{E}$ of a matrix ordered operator space $E$ in its partial unitization $\pu{E}$ is a complete isomorphism with $\lVert \phi \rVert_{cb} \leq 1$, $\lVert \phi^{-1} \rVert_{cb} \leq \kappa$ if and only if $\lVert \:\cdot\: \rVert_n$ is $\kappa$\nobreakdash-max-normal for all $n \in \N$.
This result was announced in \cite[Remark 4.14(iii)]{Werner}, but to our knowledge no proof exists in the literature.
For $\kappa = 1$ we recover a simple axiomatic characterization of abstract non-unital operator systems, which was first proved by Russell \cite[Theorem~1.1]{Russell}; see \autoref{cor:isometric-operator-system} below.

\subsubsection*{Organization of the paper}

In \mysecref{sec:preliminaries}, we recall the necessary prerequisites.
In \mysecref{sec:construction}, we construct the space $\unit{E}$ and prove that it is a well-defined AOU space.
In \mysecref{sec:universal-property}, we prove necessary and sufficient conditions for a continuous positive linear map $E \to F$ to extend to a map $\unit{E} \to F$, and we use this to prove the universal property (\autoref{thm:intro:universal-property}) and a number of related results.
In \mysecref{sec:seminorm-normalization}, we study the relation between the original seminorm $p$ of $E$ and the seminorm $p_u$ inherited from the unitization.
In \mysecref{sec:special-cases}, we work out two examples: the unitization of a (not necessarily Archimedean) order unit space, and the unitization of the self-adjoint part of a $C^*$\nobreakdash-algebra.
In \mysecref{sec:representations}, we discuss the role of unitizations in representation theory of ordered vector spaces, and we give an alternative construction of $\unit{E}$ in terms of a representation $E \to C(\Omega)$.
Finally, in \mysecref{sec:operator-systems}, we use our results to prove a number of new results about Werner's partial unitization of a matrix ordered operator space.

\bigskip
\noindent
\small
\emph{Acknowledgements.}
A rudimentary version of this paper was contained in the author's master thesis \cite[Chapter 3]{Dobben-MSc-thesis}, written at Leiden University under the supervision of Onno van Gaans and Marcel de Jeu.
The remainder of this paper was written intermittently at Delft University of Technology, supported by the Dutch Research Council (NWO), project number 613.009.127, and at the Technical University of Denmark, supported by the Carlsberg Foundation, grant number CF21-0682.
I am grateful to Onno van Gaans, Marcel de Jeu, Hent van Imhoff for many helpful comments and suggestions, and to Vern Paulsen for pointing me towards recent literature on non-unital operator systems.

\normalsize

\section{Preliminaries}
\label{sec:preliminaries}

Throughout this paper, all vector spaces will be over the real numbers, unless stated otherwise.
For a topological vector space $E$, we denote the topology by $\lintop$, the algebraic dual by $E\algdual$, and the topological dual by $E\topdual$.

\subsection{Ordered vector spaces}
\label{subsec:ovs}
Let $E$ be a vector space. A \emph{convex cone}%
	\hair\footnote{Some authors call this a \emph{wedge}, and reserve the term \emph{cone} for what we call a \emph{proper cone}.}
is a non-empty subset $\mywedge \subseteq E$ satisfying $\mywedge + \mywedge \subseteq \mywedge$ and $\lambda \mywedge \subseteq \mywedge$ for all $\lambda \in \R_{\geq 0}$.
The \emph{lineality space} of $\mywedge$ is the largest subspace contained in $\mywedge$, which is given by $\lineal(\mywedge) = \mywedge \cap -\mywedge$.
We say that $\mywedge$ is \emph{proper} if $\lineal(\mywedge) = \{0\}$, and \emph{generating} if $\mywedge - \mywedge = E$.
A \emph{base} of a convex cone $E_+$ is a convex subset $\mathcal B \subseteq E_+$ such that every element of $E_+ \setminus \{0\}$ can be written as $\lambda b$ with $\lambda > 0$ and $b \in \mathcal B$ in exactly one way.

A \emph{preordered vector space} is a tuple $(E,E_+)$, where $E$ is a real vector space and $E_+ \subseteq E$ is a convex cone.
The cone $E_+$ induces a vector preorder $\leq$ on $E$, where $x \leq y$ if and only if $y - x \in E_+$.
The preorder $\leq$ is a partial order if and only if $E_+$ is proper, in which case $(E,E_+)$ is called an \emph{ordered vector space}.

Let $(E,E_+)$ be a preordered vector space.
Given $x,y \in E$, the \emph{order interval} $[x,y]$ is the set $\{z \in E \, : \, x \leq z \leq y\}$.
A subset $M \subseteq E$ is called \emph{full} (or \emph{order-convex}) if $[x,y] \subseteq M$ whenever $x,y \in M$.
The \emph{full hull} $\fh(M)$ of a set $M \subseteq E$ is the smallest full set containing $M$, and is given by $\fh(M) = \bigcup_{x,y \in M} [x,y]$.

A \emph{preordered topological vector space} $(E,E_+,\lintop)$ is a preordered vector space $(E,E_+)$ equipped with a linear topology $\lintop$, where we make no assumptions on compatibility between $E_+$ and $\lintop$.
If $\lintop$ has a neighbourhood basis at $0$ consisting of $E_+$\nobreakdash-full sets, then we say that the topology $\lintop$ is \emph{locally full} (with respect to $E_+$) and that the cone $E_+$ is \emph{normal} (with respect to $\lintop$).

Following \cite{Messerschmidt}, we say that a seminorm $p$ on $E$ is \emph{$\alpha$\nobreakdash-max-normal} if $x \leq y \leq z$ implies $p(y) \leq \alpha \max(p(x),p(z))$, and \emph{monotone} if $0 \leq x \leq y$ implies $p(x) \leq p(y)$.
Clearly every $1$\nobreakdash-max-normal seminorm is monotone.
Conversely, it follows from the proof of \cite[Theorem 2.38]{Aliprantis-Tourky} that every monotone seminorm is $3$\nobreakdash-max-normal (see also \cite[Proposition 1]{Messerschmidt}).
It is well known that $E_+$ is normal with respect to $p$ if and only if $p$ is equivalent to an $\alpha$\nobreakdash-max-normal seminorm for some $\alpha \geq 1$; see for instance \cite[Theorem 2.38]{Aliprantis-Tourky}.

A linear map $T : E \to F$ between preordered vector spaces $(E,E_+)$ and $(F,F_+)$ is called \emph{positive} if $T[E_+] \subseteq F_+$, and \emph{bipositive} if $E_+ = T^{-1}[F_+]$.
An \emph{order embedding} is an injective bipositive map, and an \emph{order isomorphism} is a surjective order embedding.

\subsection{Archimedean order unit spaces}
\label{subsec:AOU}
For $x,y \in E$, denote by $\cosegment{x}{y}$ the half-open line segment $\{ (1 - t)x + t y \, : \, t \in [0,1)\}$ joining $x$ and $y$.
Following \cite[\S 16.2]{Kothe-I}, we say that a convex set $M$ is \emph{algebraically closed} if for all $x,y \in E$ such that $\cosegment{x}{y} \subseteq M$, one has $y \in M$.

A preordered vector space $(E,E_+)$ and its positive cone $E_+$ are said to be \emph{Archimedean} if $E_+$ is proper and for every $x \in E$ it holds that, if $\{nx \, : \, n \in \N_1\}$ is bounded above, then $x \leq 0$.
An equivalent formulation of the Archimedean property is given by the following proposition, whose simple proof we omit.
\begin{proposition}
	\label{prop:algebraically-closed}
	A convex cone is Archimedean if and only if it is proper and algebraically closed.
\end{proposition}

An element $u \in E$ is called an \emph{order unit} if for every $x \in E$ there exists some $\lambda > 0$ such that $-\lambda u \leq x \leq \lambda u$.
If $u \in E$ is an order unit, then the function
\begin{equation}
	\mu_u(x) = \inf \{ \lambda \in \R_{\geq 0} \, : \, -\lambda u \leq x \leq \lambda u \} \label{eqn:order-unit}
\end{equation}
defines a seminorm on $E$, called the \emph{order unit seminorm}.
It is easy to see that $\mu_u$ is $1$\nobreakdash-max-normal.

The following proposition will be helpful in obtaining a simple formula for $\mu_u$.

\begin{proposition}[{compare \cite[Lemma 2.3]{Kadison-representation}}]
	\label{prop:order-unit-seminorm}
	Let $(E,E_+)$ be a preordered vector space containing an order unit $u \in E_+$ such that $-u \notin E_+$.
	Then for every $x \in E$ the numbers
	\begin{align*}
		\alpha_u(x) &:= \sup\{\alpha \in \R \, : \, \alpha u \leq x\},\\
		\omega_u(x) &:= \inf\{\omega \in \R \, : \, x \leq \omega u\}
	\end{align*}
	are well-defined and satisfy $\mu_u(x) = \max(-\alpha_u(x),\omega_u(x))$.
\end{proposition}
\begin{proof}
	Choose $\lambda_0 > 0$ such that $-\lambda_0 u \leq x \leq \lambda_0 u$.
	Then the set $\{\omega \in \R \, : \, x \leq \omega u\}$ is non-empty (it contains $\lambda_0$) and bounded below by $-\lambda_0$ (after all, if $-\lambda_0 u \leq x \leq \omega u$, then $\omega \geq -\lambda_0$ because $-u \notin E_+$).
	This shows that $\omega_u(x)$ is well-defined.
	Analogously, $\alpha_u(x)$ is also well-defined.
	
	To prove the formula for $\mu_u(x)$, note that
	\[ \{\lambda \in \R \, : \, -\lambda u \leq x \leq \lambda u\} = \{\alpha \in \R \, : \, -\alpha u \leq x\} \cap \{\omega \in \R \, : \, x \leq \omega u\}. \]
	Since each of these sets is upwards closed, we have
	\[ \inf \{\lambda \in \R \, : \, -\lambda u \leq x \leq \lambda u\} = \max(-\alpha_u(x) , \omega_u(x)). \]
	Moreover, since the set in the left-hand side does not contain negative numbers (because $-u \notin E_+$), the infimum may also be taken over $\lambda \in \R_{\geq 0}$ instead of $\lambda \in \R$.
\end{proof}

An \emph{order unit space} is a triple $(E,E_+,u)$, where $(E,E_+)$ is a preordered vector space and $u \in E_+$ is a designated order unit.
The \emph{state space} of an order unit space $(E,E_+,u)$ is the set $\{ \varphi \in E_+\algdual \, : \, \varphi(u) = 1\}$.

An \emph{Archimedean order unit \textup(AOU\textup) space} is an order unit space $(E,E_+,u)$ for which $E_+$ is Archimedean.
If $(E,E_+,u)$ is an AOU space, then $\mu_u$ is a norm, which we denote by $\lVert \:\cdot\: \rVert_u$.
In this case, the infimum in \eqref{eqn:order-unit} is attained, the closed unit ball of $\lVert \:\cdot\: \rVert_u$ coincides with the order interval $[-u,u]$, and the positive cone $E_+$ is $\lVert \:\cdot\: \rVert_u$\nobreakdash-closed.
See for instance \cite[Theorem 2.55]{Aliprantis-Tourky}.

A positive linear map $T : E \to F$ between order unit spaces $(E,E_+,u)$ and $(F,F_+,w)$ is called \emph{unital} if $T(u) = w$.

\section{Construction of the unitization}
\label{sec:construction}

In this section, we give the concrete construction of the minimal Archimedean order unitization, and we establish the basic properties of this space.
We will prove in \mysecref{sec:universal-property} that this space satisfies the universal property from \autoref{thm:intro:universal-property}.

\begin{definition}
	\label{def:minimal-unitization}
	Let $(E,E_+,p)$ be a seminormed preordered vector space.
	Let $\overline{E_+}$ denote the closure of $E_+$ with respect to the seminorm $p$, and let $N := \lineal(\overline{E_+}) = \overline{E_+} \cap -\overline{E_+}$ denote its lineality space.
	The \emph{Archimedean order unitization} of $E$ is the triple $(\tilde E, \tilde E_+, u)$, where
	\begin{align*}
		\tilde E &:= E/N \oplus \R; \\
		\tilde E_+ &:= \big\{(x + N,\lambda) \in \tilde E \: : \: \lambda \geq d(x,E_+)\big\}; \\\noalign{\smallskip}
		u &:= (0 + N, 1).
	\end{align*}
	Here $d(x,E_+) = \inf_{y\in E_+} p(x - y)$ denotes the distance (in $E$) between $x$ and $E_+$.
\end{definition}

We will prove in \autoref{prop:convex-cone} and \autoref{prop:adjoined-order-unit} that $\tilde E_+$ is well-defined and that $(\tilde E, \tilde E_+, u)$ is indeed an AOU space.
The proof that this is indeed the Archimedean order unitization from \autoref{thm:intro:universal-property} is postponed until \mysecref{sec:universal-property}.

First, to illustrate the construction, consider the following two examples.

\begin{example}
	\label{xmpl:ice-cream-cone}
	If $E_+ = \{0\}$ and $p$ is a norm, then $\tilde E_+ = \{(x, \lambda) \in E \oplus \R \, : \, \lVert x \rVert \leq \lambda\}$ is the ``ice cream cone'' for the norm $p$.
\end{example}

\begin{example}
	\label{xmpl:R2-lq}
	Let $(E,E_+,p) = (\R^2,\R_{\geq 0}^2,\lVert \:\cdot\: \rVert_{\ell_q})$, where $\lVert \:\cdot\: \rVert_{\ell_q}$ denotes the $\ell_q$\nobreakdash-norm for $1 \leq q \leq \infty$.
	Then $N = \{0\}$, so we have $\tilde E = \R^2 \oplus \R$, and the set $\mathcal B = \{(x,1) \in \R^2 \oplus \{1\} \, : \, d(x,E_+) \leq 1\}$ is a base for the cone $\tilde E_+$.
	The shape of this base depends on the value of $q$, as illustrated in \autoref{fig:R2-example}.
	It is not hard to see that the resulting cone $\tilde E_+$ is polyhedral if and only if $q = 1$ or $q = \infty$, and a lattice cone if and only if $q = \infty$ (use \cite[Theorem 3.21]{Aliprantis-Tourky}).
	This shows that the properties of $\tilde E_+$ depend not only on $E_+$ and the topology of $E$, but also on the seminorm.
	\begin{figure}[h!t]
		\setlength{\belowcaptionskip}{-4pt}
		\centering
		\begin{tikzpicture}[scale=.65,
			            cone_style/.style={red,fill opacity=0.5},
			            cone_border_style/.style={red!70!black,thick},
			            cone_label/.style={red!30!black}]
			\begin{scope}[xshift=-6cm]
				\draw[gray!50,thin] (-2,-2) grid (2,2);
				\fill[cone_style] (0,2) -- (-1,2) -- (-1,0) -- (0,-1) -- (2,-1) |- (0,2);
				\draw[thick] (-2,0) -- (2,0) (0,-2) -- (0,2);
				\draw[cone_border_style] (-1,2) -- (-1,0) -- (0,-1) -- (2,-1);
				\node[cone_label] at (1,1) {$\mathcal B$};
				\node[anchor=base] at (0,-3) {$q = 1$};
			\end{scope}
			\begin{scope}[xshift=0cm]
				\draw[gray!50,thin] (-2,-2) grid (2,2);
				\fill[cone_style] (0,2) -- (-1,2) -- (-1,0) arc (180:270:1cm) -- (2,-1) |- (0,2);
				\draw[thick] (-2,0) -- (2,0) (0,-2) -- (0,2);
				\draw[cone_border_style] (-1,2) -- (-1,0) arc (180:270:1cm) -- (2,-1);
				\node[cone_label] at (1,1) {$\mathcal B$};
				\node[anchor=base] at (0,-3) {$q = 2$};
			\end{scope}
			\begin{scope}[xshift=6cm]
				\draw[gray!50,thin] (-2,-2) grid (2,2);
				\fill[cone_style] (0,2) -- (-1,2) -- (-1,-1) -- (2,-1) |- (0,2);
				\draw[thick] (-2,0) -- (2,0) (0,-2) -- (0,2);
				\draw[cone_border_style] (-1,2) -- (-1,-1) -- (2,-1);
				\node[cone_label] at (1,1) {$\mathcal B$};
				\node[anchor=base] at (0,-3) {$q = \infty$};
			\end{scope}
		\end{tikzpicture}
		\caption{The base of the positive cone of the minimal Archimedean order unitization of $(\R^2,\R_{\geq 0}^2,\lVert \:\cdot\: \rVert_{\ell_q})$ in the $z = 1$ plane, for different values of $q$.}
		\label{fig:R2-example}
	\end{figure}
\end{example}

We proceed to show that $(\tilde E, \tilde E_+, u)$ is a (well-defined) AOU space, and we study the basic properties of this space.

\pagebreak
\begin{proposition}
	\label{prop:convex-cone}
	Let $(E,E_+,p)$ be a seminormed preordered vector space.
	Then the set $\tilde E_+$ from \autoref{def:minimal-unitization} is a well-defined, proper convex cone.
\end{proposition}
\begin{proof}
	Since $E_+$ is a convex cone, for all $x,y \in E$ and all $\alpha > 0$ we have
	\begin{align}
		\label{eqn:well-defined:alpha-x} d(\alpha x,E_+) &= \inf_{y\in E_+} \lVert \alpha x - y\rVert = \inf_{z\in E_+} \lVert \alpha x - \alpha z\rVert = \alpha\, d(x,E_+); \\\noalign{\smallskip}
		\nonumber d(x + y,E_+) &= \inf_{v,w\in E_+} \lVert x + y - v - w\rVert\\
		\label{eqn:well-defined:x+y} &\leq \inf_{v,w\in E_+} \lVert x - v\rVert + \lVert y - w\rVert\\
		\nonumber &= d(x,E_+) + d(y,E_+).
	\end{align}
	To show that $\tilde E_+$ is well-defined, note that we have $\overline{E_+} = \{z \in E \, : \, d(z,E_+) = 0\}$, and therefore $N = \{z \in E \, : \, d(z,E_+) = d(-z,E_+) = 0\}$.
	Consequently, by \eqref{eqn:well-defined:x+y}, for all $x \in E$, $z \in N$ we have
	\[ d(x, E_+) \leq d(x + z, E_+) + d(-z, E_+) = d(x + z, E_+) \leq d(x, E_+) + d(z, E_+) = d(x, E_+), \]
	hence $d(x,E_+) = d(x + z, E_+)$.
	This shows that $d(x, E_+) = d(x', E_+)$ whenever $x + N = x' + N$, so $\tilde E_+$ is well-defined.
	
	Using \eqref{eqn:well-defined:alpha-x} and \eqref{eqn:well-defined:x+y}, it is easy to see that $\tilde E_+$ is a convex cone.
	
	To show that $\tilde E_+$ is proper, let $(x + N,\lambda) \in \tilde E_+ \cap -\tilde E_+$ be given.
	Then we have $0 \leq d(x,E_+) \leq \lambda$ as well as $0 \leq d(-x,E_+) \leq -\lambda$, so we find $\lambda = 0$ as well as $d(x,E_+) = d(-x,E_+) = 0$. From the latter it follows that $x,-x \in \overline{E_+}$, so $x \in N$.
	Therefore $(x + N,\lambda) = (0 + N , 0)$, which shows that $\tilde E_+$ is proper.
\end{proof}

\begin{proposition}
	\label{prop:adjoined-order-unit}
	Let $(E,E_+,p)$ be a seminormed preordered vector space.
	Then the cone $\tilde E_+$ defined in \autoref{def:minimal-unitization} is Archimedean, and $(0 + N,1) \in \tilde E_+$ is an order unit.
\end{proposition}
\begin{proof}
	First, we claim that the function $E \to \R$, $x \mapsto d(x,E_+)$ is uniformly continuous. Indeed, if $x,y\in E$ lie distance $\varepsilon$ apart, then $d(y,E_+)$ is at most $d(x,E_+) + \varepsilon$, by the triangle inequality. Analogously, we have $d(x,E_+) \leq d(y,E_+) + \varepsilon$, so we find $|d(x,E_+) - d(y,E_+)| \leq p(x - y)$, which proves our claim.
	
	Let $(x + N, \lambda),(y + N, \mu) \in E/N \oplus \R$ be such that the half-open line segment $\cosegment{(x + N, \lambda)}{(y + N, \mu)}$ (see \mysecref{subsec:AOU}) is contained in $\tilde E_+$.
	Then for all $t \in [0,1)$ we have $(1 - t)\lambda + t\mu - d((1 - t)x + ty , E_+) \geq 0$.
	By continuity, it follows that $\mu - d(y,E_+) \geq 0$, so $(y + N, \mu) \in \tilde E_+$.
	This shows that $\tilde E_+$ is algebraically closed, so it follows from \autoref{prop:algebraically-closed} (together with \autoref{prop:convex-cone}) that $\tilde E_+$ is Archimedean.
	
	To prove that $(0 + N, 1) \in \tilde E$ is an order unit, note that for all $x \in E$ and all $\lambda \in \R$ we have $p(x) + |\lambda| - \lambda \geq p(x) = p(-x) \geq d(-x,E_+)$ (since $0 \in E_+$), hence $(x + N, \lambda) \leq (p(x) + |\lambda|) \cdot (0 + N, 1)$.
	Therefore also $(-x + N, -\lambda) \leq (p(-x) + |-\lambda|) \cdot (0 + N, 1) = (p(x) + |\lambda|) \cdot (0 + N, 1)$, so we have
	\begin{equation}
		-(p(x) + |\lambda|) \cdot (0 + N, 1) \leq (x + N, \lambda) \leq (p(x) + |\lambda|) \cdot (0 + N, 1). \label{eqn:norm-decreasing}
	\end{equation}
	This shows that $(0 + N, 1)$ is an order unit.
\end{proof}

The preceding results show that $(\tilde E,\tilde E_+,u)$ is an AOU space.
We will henceforth understand $\tilde E$ to be normed with the order unit norm $\lVert \:\cdot\: \rVert_u$ (see \mysecref{subsec:AOU}).
Using \autoref{prop:order-unit-seminorm}, we find the following simple formula for this norm.

\begin{proposition}
	\label{prop:simple-formula}
	Let $(E,E_+,p)$ be a seminormed preordered vector space, and let $\tilde E$ be as in \autoref{def:minimal-unitization}.
	Then the functions $\alpha_u,\omega_u : \tilde E \to \R$, as defined in \autoref{prop:order-unit-seminorm}, are given by
	\[ \alpha_u((x + N, \lambda)) = \lambda - d(x,E_+) \qquad\text{and}\qquad \omega_u((x + N, \lambda)) = \lambda + d(-x,E_+). \]
	Consequently, the order unit norm $\lVert \:\cdot\: \rVert_u$ of $\tilde E$ is given by
	\[ \lVert (x + N, \lambda)\rVert_u = \max\big( \hair d(x,E_+) - \lambda \, , \, d(-x, E_+) + \lambda \hair \big). \]
\end{proposition}
\begin{proof}
	By definition, we have
	\begin{align*}
		\omega_u((x + N, \lambda)) &= \inf \{ \omega \in \R \: : \: (x + N, \lambda) \leq (0 + N, \omega) \}\\
		&= \inf \{ \omega \in \R \: : \: (-x + N, \omega - \lambda) \in \tilde E_+ \}\\
		&= \inf \{ \omega \in \R \: : \: \omega - \lambda \geq d(-x,E_+) \}\\
		&= \lambda + d(-x,E_+).
	\end{align*}
	The formula for $\alpha_u((x,\lambda))$ follows analogously, and the formula for $\lVert \:\cdot\: \rVert_u$ follows immediately from \autoref{prop:order-unit-seminorm}.
\end{proof}

We conclude this section with the following basic properties of the map $\phi : E \to \tilde E$.

\begin{proposition}
	\label{prop:order-unitization-inj-pos}
	Let $(E,E_+,p)$ be a seminormed preordered vector space.
	Then the natural map $\phi : E \to \tilde E$, $x \mapsto (x + N,0)$ is positive and contractive, and $\phi^{-1}[\tilde E_+] = \overline{E_+}$.
\end{proposition}
\begin{proof}
	If $x \in E_+$, then $d(x, E_+) = 0$, so $(x + N,0) \in \tilde E_+$.
	This shows that $\phi$ is positive.
	Moreover, it follows from \eqref{eqn:norm-decreasing} that $\lVert (x + N, \lambda) \rVert_u \leq p(x) + |\lambda|$, so in particular $\lVert \phi(x) \rVert_u \leq p(x)$.
	This shows that $\phi$ is contractive.
	Finally, by definition we have $\phi(x) \in \tilde E_+$ if and only if $d(x,E_+) = 0$, so we have $\phi^{-1}[\tilde E_+] = \overline{E_+}$.
\end{proof}

\begin{corollary}
	The natural map $\phi : E \to \tilde E$ is bipositive if and only if $E_+$ is closed.
\end{corollary}

Note that a bipositive map from a preordered vector space to an AOU space is not necessarily injective.
The natural map $\phi : E \to \tilde E$ is injective if and only if $\overline{E_+}$ is a proper cone.

\section{Extensions of positive continuous linear maps}
\label{sec:universal-property}

In this section, we consider the following question: when can a positive continuous linear map $\psi : E \to F$ be extended to a positive continuous linear map $\tilde E \to F$?
We show that this can be done whenever $F$ is an AOU space.
Using this, we show that $\tilde E$ satisfies the universal property from \autoref{thm:intro:universal-property}, and we give a few additional properties of the unitization.

The main technique in this section is the following construction.

\begin{lemma}
	\label{lem:extension}
	Let $(E,E_+,p)$ be a seminormed preordered vector space, let $(F,F_+,w)$ be an AOU space, let $\psi : E \to F$ be a positive continuous linear map, and let $\alpha \in \R$.
	Then the linear map $\chi_\alpha : \tilde E \to F$, $(x + N,\lambda) \mapsto \psi(x) + \alpha\lambda w$ is well-defined, and it is positive if and only if $\alpha \geq \lVert \psi \rVert$.
	If this is the case, then $\chi_\alpha$ is continuous with $\lVert \chi_\alpha \rVert = \alpha$.
\end{lemma}
\begin{proof}
	To see that $\chi_\alpha$ is well-defined, note that $F_+$ is closed (since $F$ is an AOU space), so by continuity and positivity we have $\psi[\overline{E_+}] \subseteq F_+$.
	Hence for $x \in N = \overline{E_+} \cap -\overline{E_+}$ we have $\pm \psi(x) \in F_+$, hence $\psi(x) = 0$ because $F_+ \cap -F_+ = \{0\}$.
	This shows that $N \subseteq \ker(\psi)$, so $\chi_\alpha$ is well-defined.
	
	By basic properties of AOU spaces, every positive linear map $\tilde\psi : \tilde E \to F$ is automatically continuous with $\lVert \tilde\psi \rVert = \lVert \tilde\psi(u) \rVert_w$.
	Hence, if $\chi_\alpha$ is positive, then $\lVert \chi_\alpha \rVert = \alpha$.
	
	Suppose that $\chi_\alpha$ is positive.
	Then, by the preceding paragraph, $\chi_\alpha$ is continuous with $\lVert \chi_\alpha \rVert = \alpha$.
	Furthermore, we recall from \autoref{prop:order-unitization-inj-pos} that $\phi : E \to \tilde E$ is continuous with $\lVert \phi \rVert \leq 1$.
	Since $\psi = \chi_\alpha \circ \phi$, it follows that $\lVert \psi \rVert \leq \lVert \chi_\alpha \rVert \cdot \lVert \phi \rVert \leq \lVert \chi_\alpha \rVert = \alpha$.
	
	Conversely, suppose that $\alpha \geq \lVert \psi \rVert$.
	Let $(x + N,\lambda) \in \tilde E_+$.
	Since $\psi[E_+] \subseteq F_+$, we have
	\begin{align*}
		d_w(\psi(x),F_+) &= \inf_{y\in F_+} \lVert \psi(x) - y \rVert_w\\
		&\leq \inf_{z\in E_+} \lVert \psi(x) - \psi(z) \rVert_w\\
		&\leq \lVert \psi \rVert \cdot d(x,E_+)\\
		&\leq \lVert \psi \rVert \cdot \lambda.
	\end{align*}
	Since the closed unit ball of $\lVert \:\cdot\: \rVert_w$ coincides with the order interval $[-w,w]$, it follows that for all $\mu > \lambda \lVert \psi \rVert$ we have $(\psi(x) + [-\mu w,\mu w]\hair) \cap F_+ \neq \varnothing$, and therefore $\psi(x) + \mu w \geq 0$.
	Since $F_+$ is Archimedean, it follows from \autoref{prop:algebraically-closed} that $F_+$ is algebraically closed, so we also have $\psi(x) + \lambda\lVert \psi \rVert w  \geq  0$.
	This shows that $\chi_\alpha$ is positive whenever $\alpha \geq \lVert \psi \rVert$.
\end{proof}

Using this lemma, we can now prove our main result.

\begin{theorem}[{$=$ \autoref{thm:intro:universal-property}}]
	\label{thm:universal-property}
	For every seminormed preordered vector space $E$ there is an AOU space $\unit{E}$ and a contractive positive linear map $\phi : E \to \unit{E}$ with the following universal property:
	for every AOU space $F$ and every contractive positive linear map $\psi : E \to F$, there is a unique unital positive linear map $\mUnit{\psi} : \unit{E} \to F$ such that $\psi = \mUnit{\psi} \circ \phi$.
	\[ \begin{tikzcd}[column sep={2.5em,between origins}]
		& E\arrow[dl,"\phi",swap]\arrow[dr,"\psi"] \\
		\unit{E}\arrow[rr,"\mUnit{\psi}",densely dashed] & & F
	\end{tikzcd} \]
	By the universal property, $\unit{E}$ is unique up to a unique unital order isomorphism.
\end{theorem}
\begin{proof}
	We show that the AOU space $(\tilde E, \tilde E_+,u)$ from \autoref{def:minimal-unitization} satisfies the universal property.
	Let $\phi : E \to \tilde E$ be the natural map $x \mapsto (x + N, 0)$.
	For every AOU space $(F,F_+,w)$ and every contractive positive linear map $\psi : E \to F$, we define $\mUnit{\psi} : \tilde E \to F$ by $(x + N,\lambda) \mapsto \psi(x) + \lambda w$.
	It follows from \autoref{lem:extension} (with $\alpha = 1$) that $\mUnit{\psi}$ is well-defined and positive, and it is easy to see that $\psi = \mUnit{\psi} \circ \phi$ and $\mUnit{\psi}(u) = w$.
	Therefore $\mUnit{\psi}$ is a unital positive linear extension of $\psi$.
	
	To show that $\mUnit{\psi}$ is unique, let $\chi : \tilde E \to F$ be any unital positive linear extension of $\psi$.
	Then $\chi(x + N, 0) = \psi(x)$ for all $x \in E$ (because $\psi = \chi \circ \phi$) and $\chi(0 + N, 1) = w$ (because $\chi$ is unital), hence $\chi(x + N, \lambda) = \psi(x) + \lambda w$, by linearity.
	This shows that $\chi = \mUnit{\psi}$, so the extension is unique.
	
	By a standard arrow-theoretic argument, it follows from the universal property that $\unit{E}$ is unique up to a unique unital order isomorphism.
\end{proof}

\begin{remark}
	We point out that the universal (minimal) Archimedean order unitization $\unit{E}$ is always equal to $\tilde E = E/N \oplus \R$, even if $E$ is already an AOU space.
	We need the extra degree of freedom to satisfy the universal property, as we need to be able to extend every \emph{contractive} positive linear map $\psi : E \to F$ to a \emph{unital} positive linear map $\mUnit{\psi} : \unit{E} \to F$.
	If $E$ is an AOU space, then $E$ itself does not satisfy the universal property from \autoref{thm:universal-property}, as can be seen by taking $\psi = \lambda \id_E$ for some $\lambda \in [0,1)$.
	Now $\psi$ is positive and contractive, but it cannot be extended to a \emph{unital} positive linear map $E \to E$.
\end{remark}

We proceed to look at some consequences and limitations of the extension lemma.

First, instead of requiring the extension to be unital, we can also find an extension of the same norm.

\begin{corollary}
	\label{cor:extension}
	Let $(E,E_+,p)$ be a seminormed preordered vector space, and let $(F,F_+,w)$ be an AOU space.
	Then every positive continuous linear map $\psi : E \to F$ has a positive extension $\tilde\psi : \tilde E \to F$ of the same norm.
	Furthermore, there is exactly one such extension with the additional property that $u$ is mapped to a multiple of $w$.
\end{corollary}
\begin{proof}
	It follows from \autoref{lem:extension} (with $\alpha = \lVert \psi \rVert$) that $\chi_{\lVert \psi \rVert} : \tilde E \to F$, $(x + N,\lambda) \mapsto \psi(x) + \lambda \lVert \psi \rVert w$ is a (well-defined) positive extension of $\psi$ of the same norm, which furthermore maps $u$ to a multiple of $w$.
	
	Now let $\tilde \psi : \tilde E \to F$ be an arbitrary positive extension of $\psi$ such that $\lVert \tilde \psi \rVert = \lVert \psi \rVert$ and $\tilde \psi(u) = \alpha w$ for some $\alpha \in \R$.
	Then $\tilde \psi = \chi_\alpha$, and since $\tilde \psi$ is positive it follows from \autoref{lem:extension} that $\lVert \tilde \psi \rVert = \alpha$.
	Therefore $\lVert \psi \rVert = \lVert \tilde \psi \rVert = \alpha$, which shows that $\tilde \psi = \chi_{\lVert \psi \rVert}$.
	This proves uniqueness.
\end{proof}

The following example shows that a positive extension of the same norm is no longer unique if we do not require that $u$ is mapped to a multiple of $w$.

\begin{example}
	For $n \in \N$, let $(\R^n,\R^n_+,\one)$ be the AOU space with standard cone $\R_+^n = \R_{\geq 0}^n$ and order unit $\one = (1,\ldots,1) \in \R^n$.
	The order unit norm of this space is just the $\ell_\infty$\nobreakdash-norm.
	Furthermore, let $\unit{(\R^n)} = \widetilde{\R^n}$ denote the Archimedean order unitization of this space.
	
	Let $\iota : \R^1 \to \R^2$ be the embedding $x \mapsto (x,0)$.
	We claim that $\iota$ has more than one positive extension $\unit{(\R^1)} \to \R^2$ of norm $1$.
	Note that $\unit{(\R^1)} = \R^2$ with cone $\{ (x,\lambda) \, : \, \lambda \geq 0, \ x + \lambda \geq 0\}$ and order unit $(0,1)$.
	The canonical extension of $\iota$ is given by $\mUnit{\iota}(x,\lambda) = (x + \lambda , \lambda)$, but the map $(x,\lambda) \mapsto (x + \lambda , 0)$ is also a positive extension of norm $1$.
\end{example}

For positive continuous linear functionals, we do have uniqueness.

\begin{corollary}
	Let $(E,E_+,p)$ be a seminormed preordered vector space.
	If $\varphi : E \to \R$ is a positive continuous linear functional, then there is a unique positive continuous linear extension $\tilde \varphi : \tilde E \to \R$ of the same norm, given by $\tilde \varphi(x + N, \lambda) = \varphi(x) + \lambda \lVert \varphi \rVert$.
\end{corollary}
\begin{proof}
	Since every map $\tilde E \to \R$ necessarily maps everything to a multiple of the order unit $1 \in \R$, this follows from \autoref{cor:extension}.
\end{proof}

Another immediate consequence of \autoref{lem:extension} is that every positive continuous linear map $E \to F$ can be extended to a positive continuous linear map $\tilde E \to \tilde F$.

\begin{corollary}
	Let $(E,E_+,p)$, $(F,F_+,q)$ be seminormed preordered vector spaces.
	If $\psi : E \to F$ is a positive continuous linear map, then the map $\tilde \psi : \tilde E \to \tilde F$ given by $(x + N,\lambda) \mapsto (\psi(x) + M, \lambda \lVert \psi \rVert)$ is well-defined, positive, and continuous, with $\lVert \tilde \psi \rVert = \lVert \psi \rVert$.
\end{corollary}\vspace*{-5mm}
\begin{proof}
	Applying \autoref{cor:extension} to the composition $E \stackrel{\psi}{\longrightarrow} F \stackrel{\big.\phi\big._F}{\longrightarrow} \tilde F$ shows that $\tilde \psi : \tilde E \to \tilde F$, $(x + N,\lambda) \mapsto (\psi(x) + M , \lambda \lVert \psi \rVert)$ is a well-defined positive and continuous extension of $\phi_F \circ \psi$, with $\lVert \tilde \psi \rVert = \lVert \phi_F \circ \psi \rVert \leq \lVert \psi \rVert$.
	Since we have $\tilde \psi(0 + N,1) = (0 + M,\lVert \psi \rVert)$, it is clear that $\lVert \tilde \psi \rVert = \lVert \psi \rVert$.
\end{proof}

Next, as another application of \autoref{lem:extension}, we compare the original seminorm $p$ of $E$ with the seminorm $p_u(x) = \lVert \phi(x) \rVert_u$ it inherits from $\tilde E$.
By \autoref{prop:simple-formula}, $p_u$ is given by
\[ p_u(x) = \max\big(d(x,E_+),d(-x,E_+)\big). \]
Clearly $p_u(x) \leq p(x)$, but the two seminorms are generally not equivalent (see \autoref{cor:equivalent-normalization} below).
We will study the seminorm $p_u$ in more detail in \mysecref{sec:seminorm-normalization}.

If $F$ is a seminormed space, then a continuous linear map $\psi : (E,p_u) \to F$ is also continuous as a map $(E,p) \to F$.
Let $\lVert \psi \rVert_u$ and $\lVert \psi \rVert$ denote the norm of $\psi$ as a map $(E,p_u) \to F$ and $(E,p) \to F$, respectively.
Then clearly $\lVert \psi \rVert \leq \lVert \psi \rVert_u$.
Remarkably, for positive continuous linear maps to an AOU space, we have the following converse.

\begin{corollary}
	\label{cor:extension-comparison}
	Let $(E,E_+,p)$ be a seminormed preordered vector space, and let $(F,F_+,w)$ be an AOU space.
	Then a linear map $\psi : E \to F$ is positive and continuous as a map $(E,E_+,p) \to F$ if and only if it is positive and continuous as a map $(E,\overline{E_+},p_u) \to F$.
	Furthermore, if this is the case, then one has $\lVert \psi \rVert = \lVert \psi \rVert_u$.
\end{corollary}

Here $\overline{E_+}$ denotes the closure of $E_+$ with respect to the original seminorm $p$ of $E$.
Although we don't need this now, we point out that $\overline{E_+}$ coincides with the closure of $E_+$ with respect to $p_u$; see \autoref{rmk:normalization-closure} below.

\begin{proof}[{Proof of \autoref{cor:extension-comparison}}]
	Since the identity $(E,E_+,p) \to (E,\overline{E_+},p_u)$ is positive and contractive, it is clear that a positive continuous linear map $(E,\overline{E_+},p_u) \to F$ is also positive and continuous as a map $(E,E_+,p) \to F$.
	Furthermore, in this case we clearly have $\lVert \psi \rVert \leq \lVert \psi \rVert_u$.
	
	Conversely, if $\psi$ is positive and continuous as a map $(E,E_+,p) \to F$, then it follows from \autoref{cor:extension} that $\psi$ can be extended to a positive continuous linear map $\tilde \psi : \tilde E \to F$ with $\lVert \psi \rVert = \lVert \tilde \psi \rVert$.
	Since the natural map $\phi : E \to \tilde E$ defines a bipositive isometry $(E,\overline{E_+},p_u) \to \tilde E$, it follows that $\psi$ is also positive and continuous as a map $(E,\overline{E_+},p_u) \to F$, and that we have $\lVert \psi \rVert_u = \lVert \tilde \psi \circ \phi \rVert_u \leq \lVert \tilde \psi \rVert = \lVert \psi \rVert$.
\end{proof}

The special case $F = \R$ is worth mentioning: it follows that $(E,E_+,p)$ and $(E,\overline{E_+},p_u)$ have the same set of positive continuous linear functionals, which we may unambiguously denote as $E_+\topdual$.
This is remarkable, because $p$ and $p_u$ are generally not equivalent (see \autoref{cor:equivalent-normalization} below).
As a consequence, it is easy to determine the topological dual spaces of $(E,p_u)$ and $\tilde E$.

\begin{corollary}
	\label{cor:dual-space-pu}
	Let $(E,E_+,p)$ be a seminormed preordered vector space, and let $p_u(x) = \lVert \phi(x) \rVert_u$ be the seminorm inherited from $\tilde E$.
	Then $(E,p_u)\topdual = \spn(E_+\topdual)$.
\end{corollary}
\begin{proof}
	Since $\lVert \:\cdot\: \rVert_u : \tilde E \to \R$ is monotone and $\phi : E \to \tilde E$ is positive, it follows that $p_u$ is monotone.
	Therefore $E_+$ is normal with respect to $p_u$.
	Since $(E,E_+,p)$ and $(E,\overline{E_+},p_u)$ have the same positive continuous linear functionals (by the preceding remarks), and since the dual cone of a normal cone is generating (see \cite[Theorem 2.26]{Aliprantis-Tourky}), it follows that $(E,p_u)\topdual = \spn(E_+\topdual)$.
\end{proof}

\begin{corollary}
	\label{cor:unitization-dual-space}
	Let $(E,E_+,p)$ be a seminormed preordered vector space.
	Then the dual space of $\tilde E$ is isomorphic with $\spn(E_+\topdual) \oplus \R$, where $(\varphi,\alpha)(x + N, \lambda) = \varphi(x) + \alpha\lambda$.
	Under this identification, the dual cone of $\tilde E_+$ is $\{ (\varphi,\alpha) \in \spn(E_+\topdual) \oplus \R \, : \, \varphi \geq 0\ \text{and}\ \alpha \geq \lVert \varphi \rVert\}$.
\end{corollary}
\begin{proof}
	Let $\pi_2 : \tilde E = (E/N) \oplus \R \to \R$ be the projection onto the second coordinate.
	Since $\pi_2$ is positive and unital, it is continuous (with norm $1$).
	Therefore $(E/N) \oplus \{0\}$ and $\{0\} \oplus \R$ are complementary subspaces of $\tilde E$, so we have $\tilde E \cong (E/N) \oplus \R$ topologically.
	It follows that $\tilde E\topdual = (E/N)\topdual \oplus \R$.
	Since $(E,p_u)$ carries the topology inherited from $E/N$, it follows from \autoref{cor:unitization-dual-space} that $(E/N)\topdual = (E,p_u)\topdual = \spn(E_+\topdual)$.
	Thus, $\tilde E\topdual = \spn(E_+\topdual) \oplus \R$, via the pairing $(\varphi,\alpha)(x + N,\lambda) = \varphi(x) + \alpha\lambda$.
	
	If $\varphi \in E_+\topdual$ and $\alpha \geq \lVert \varphi \rVert$, then it follows from \autoref{lem:extension} that $(\varphi,\alpha)$ is positive.
	Conversely, if $(\varphi,\alpha) \in \tilde E_+\topdual$, then in particular $\varphi(x) = (\varphi,\alpha)(x + N,0) \geq 0$ for all $x \in E_+$, so $\varphi \in E_+\topdual$.
	Then, since $(\varphi,\alpha)$ is a positive extension of $\varphi$, it follows from \autoref{lem:extension} that $\alpha \geq \lVert \varphi \rVert$.
\end{proof}

\begin{corollary}
	\label{cor:state-space}
	The state space of $(\tilde E,\tilde E_+,u)$ is $\{ (\varphi,1) \in \spn(E_+\topdual) \oplus \R \, : \, \varphi \geq 0 \ \text{and}\ \lVert \varphi \rVert \leq 1\}$.
\end{corollary}

We illustrate the preceding results by revisiting the examples from \mysecref{sec:construction}.

\begin{example}[{cf.~\autoref{xmpl:ice-cream-cone}}]
	Recall: if $E_+ = \{0\}$ and $p$ is a norm, then $\tilde E_+$ is the ice cream cone $ \{(x, \lambda) \in E \oplus \R \, : \, \lVert x \rVert \leq \lambda\}$ for the norm $p$.
	We have $E_+\topdual = E\topdual$ (the dual cone is the whole space), so it follows from \autoref{cor:unitization-dual-space} that $\tilde E_+\topdual$ is the dual ice cream cone $\{ (\varphi,\alpha) \in E\topdual \oplus \R \, : \, \lVert \varphi \rVert \leq \alpha\}$.
\end{example}

\begin{example}[{cf.~\autoref{xmpl:R2-lq}}]
	Let $(E,E_+,p) = (\R^2,\R_{\geq 0}^2,\lVert \:\cdot\: \rVert_{\ell_q})$, where $\lVert \:\cdot\: \rVert_{\ell_q}$ denotes the $\ell_q$\nobreakdash-norm for $1 \leq q \leq \infty$.
	Under the natural pairing $\R^2 \times \R^2 \to \R$, we have $(\R_{\geq 0}^2)\topdual = R_{\geq 0}^2$, so the state space of $(\tilde E,\tilde E_+,u)$ is given by $\{(x,y,1) \in \R^3 \, : \, x,y \geq 0\ \text{and}\ \lVert (x,y) \rVert_{\ell_{q'}} \leq 1\}$, where $1 \leq q' \leq \infty$ is such that $\frac{1}{q} + \frac{1}{q'} = 1$.
	This is illustrated in \autoref{fig:R2-state-space}.
	\begin{figure}[h!t]
		\setlength{\belowcaptionskip}{-4pt}
		\centering
		\begin{tikzpicture}[scale=.65,
			            cone_style/.style={green,fill opacity=0.5},
			            cone_border_style/.style={green!70!black,thick},
			            cone_label/.style={green!30!black}]
			\begin{scope}[xshift=-6cm]
				\draw[gray!50,thin] (-2,-2) grid (2,2);
				\fill[cone_style] (0,0) rectangle (1,1);
				\draw[thick] (-2,0) -- (2,0) (0,-2) -- (0,2);
				\draw[cone_border_style] (0,0) rectangle (1,1);
				\node[anchor=base] at (0,-3) {$q = 1$};
			\end{scope}
			\begin{scope}[xshift=0cm]
				\draw[gray!50,thin] (-2,-2) grid (2,2);
				\fill[cone_style] (0,0) -- (1,0) arc (0:90:1cm) -- (0,0);
				\draw[thick] (-2,0) -- (2,0) (0,-2) -- (0,2);
				\draw[cone_border_style] (0,0) -- (1,0) arc (0:90:1cm) -- (0,0);
				\node[anchor=base] at (0,-3) {$q = 2$};
			\end{scope}
			\begin{scope}[xshift=6cm]
				\draw[gray!50,thin] (-2,-2) grid (2,2);
				\fill[cone_style] (0,0) -- (1,0) -- (0,1) -- (0,0);
				\draw[thick] (-2,0) -- (2,0) (0,-2) -- (0,2);
				\draw[cone_border_style] (0,0) -- (1,0) -- (0,1) -- (0,0);
				\node[anchor=base] at (0,-3) {$q = \infty$};
			\end{scope}
		\end{tikzpicture}
		\caption{The state space of the Archimedean order unitization of $(\R^2,\R_{\geq 0}^2,\lVert \:\cdot\: \rVert_{\ell_q})$, drawn in the $z = 1$ plane, for different values of $q$. (Compare \autoref{fig:R2-example}.)}
		\label{fig:R2-state-space}
	\end{figure}
\end{example}

As a final corollary, we get the following alternative formulas for the positive cone and norm of the Archimedean order unitization $\unit{E} = \tilde E$.

\begin{corollary}
	\label{cor:alternative-formulas}
	The Archimedean order unitization from \autoref{def:minimal-unitization} satisfies
	\begin{align*}
		\tilde E_+ = \big\{ (x + N, \lambda) \, : \, \varphi(x) \geq -\lambda \ \text{for all $\varphi \in E_+\topdual$ with $\lVert \varphi \rVert \leq 1$} \big\}, 
	\end{align*}
	and the order unit norm of $\tilde E$ satisfies $\lVert (x + N , \lambda) \rVert_u = \max \big\{ | \varphi(x) + \lambda | \, : \, \varphi \in E_+\topdual, \ \lVert \varphi \rVert \leq 1 \big\}$.
\end{corollary}
\begin{proof}
	Since $\tilde E_+$ is closed with respect to the order unit norm, it follows from the bipolar theorem and \autoref{cor:unitization-dual-space} that
	\begin{align*}
		\tilde E_+ &= \big\{ (x + N , \lambda) \, : \, \varphi'(x + N , \lambda) \geq 0 \ \text{for all $\varphi' \in \tilde E_+\topdual$} \big\} \\
		&= \big\{ (x + N , \lambda) \, : \, \varphi(x) + \alpha\lambda \geq 0 \ \text{for all $\varphi \in E_+\topdual$, $\alpha \geq \lVert \varphi \rVert$} \big\} \\
		&= \big\{ (x + N, \lambda) \, : \, \varphi(x) \geq -\alpha\lambda \ \text{for all $\varphi \in E_+\topdual$, $\lVert \varphi \rVert \leq \alpha$} \big\} \\
		&= \big\{ (x + N, \lambda) \, : \, \varphi(x) \geq -\lambda \ \text{for all $\varphi \in E_+\topdual$ with $\lVert \varphi \rVert \leq 1$} \big\}.
	\end{align*}
	
	\pagebreak
	\noindent
	Since $\tilde E_+$ is Archimedean, the infimum in the order unit norm \eqref{eqn:order-unit} is attained, so it follows that
	\begin{align*}
		\lVert (x + N , \lambda) \rVert_u &= \min \big \{ \mu \in \R_{\geq 0} \, : \, (0 + N , -\mu) \leq (x + N , \lambda) \leq (0 + N , \mu) \big\} \\
		&= \min \big \{ \mu \in \R_{\geq 0} \, : \, (x + N , \lambda + \mu) , (-x + N , -\lambda + \mu) \in \tilde E_+ \big\} \\
		&= \min \big \{ \mu \in \R_{\geq 0} \, : \, -\lambda - \mu \leq \varphi(x) \leq -\lambda + \mu \ \text{for all $\varphi \in E_+\topdual$ with $\lVert \varphi \rVert \leq 1$} \big\} \\
		&= \min \big \{ \mu \in \R_{\geq 0} \, : \, - \mu \leq \varphi(x) + \lambda \leq \mu \ \text{for all $\varphi \in E_+\topdual$ with $\lVert \varphi \rVert \leq 1$} \big\} \\
		&= \max \big\{ | \varphi(x) + \lambda | \, : \, \varphi \in E_+\topdual, \ \lVert \varphi \rVert \leq 1 \big\}. \qedhere
	\end{align*}
\end{proof}

\section{The 1-max-normalization of a seminorm}
\label{sec:seminorm-normalization}
Let $(E,E_+,p)$ be a seminormed preordered vector space, and let $p_u(x) = \lVert \phi(x) \rVert_u$ be the seminorm inherited from $\tilde E$ via the canonical map $\phi : E \to \tilde E$.
We already saw in the previous section that the spaces $(E,E_+,p)$ and $(E,\overline{E_+},p_u)$ have the same positive continuous linear maps to AOU spaces.
In this section, we study the relationship between $p$ and $p_u$ in more detail.

Since the natural map $\phi : E \to \tilde E$ is positive and the norm of $\tilde E$ is $1$\nobreakdash-max-normal, it is easy to see that $p_u$ defines a $1$\nobreakdash-max-normal seminorm on $E$.
We call $p_u$ the \emph{$1$\nobreakdash-max-normalization of $p$}.
It follows from \autoref{prop:simple-formula} and \autoref{cor:alternative-formulas} that $p_u$ is given by
\[ p_u(x) = \max\big(d(x,E_+),d(-x,E_+)\big) = \max \big\{ | \varphi(x) | \, : \, \varphi \in E_+\topdual, \ \lVert \varphi \rVert \leq 1 \big\}. \]

\begin{remark}
	\label{rmk:normalization-closure}
	Recall from \autoref{prop:order-unitization-inj-pos} that $\phi^{-1}[\tilde E_+] = \overline{E_+}$.
	Since $\tilde E_+$ is closed, it follows that $\overline{E_+}$ is closed not only with respect to the original seminorm $p$ but also with respect to the smaller (i.e.{} topologically coarser) seminorm $p_u$.
	This shows that the original closure $\overline{E_+}$ and the $p_u$\nobreakdash-closure $\overline{E_+}^{\,p_u}$ of $E_+$ coincide.
\end{remark}

Clearly we have $p_u(x) \leq p(x)$ for all $x\in E$.
We will show below that $p_u$ is the largest $1$\nobreakdash-max-normal seminorm smaller than $p$.
We start by comparing the unit balls.

\begin{proposition}
	\label{prop:full-hull}
	Let $(E,E_+,p)$ be a seminormed preordered vector space, and let $p_u$ be the $1$\nobreakdash-max-normalization of $p$.
	Then:
	\begin{enumerate}[label=(\alph*)]
		\item\label{itm:fh:open} The open unit ball of $p_u$ is equal to the full hull of the open unit ball of $p$;
		\item\label{itm:fh:closed} The closed unit ball of $p_u$ is equal to the $p$\nobreakdash-closure of the full hull of the closed unit ball of $p$.
	\end{enumerate}
\end{proposition}
\begin{proof}
	Let $B^\circ$ and $B_u^\circ$ denote the open unit balls of $p$ and $p_u$, respectively,%
		\hair\footnote{The notation is a bit ambiguous, because it suggests taking the interior (without specifying the topology, $p$ or $p_u$) or polar of a set. However, we believe that no confusion can arise from this.}
	and let $B$ and $B_u$ denote the corresponding closed unit balls.
	\begin{enumerate}[label=(\alph*)]
		\item Since $p_u$ is $1$\nobreakdash-max-normal, $B_u^\circ$ is full.
		Moreover, since $p_u \leq p$, we have $B^\circ \subseteq B_u^\circ$, and therefore $\fh(B^\circ) \subseteq B_u^\circ$.
		For the converse, let $y \in B_u^\circ$.
		Then $p_u(y) = \max(d(y,E_+),d(-y,E_+)) < 1$, so we may choose $x,z \in E_+$ with $p(y - x),p(y + z) < 1$.
		Since $x,z \geq 0$, it follows that $y - x \leq y \leq y + z$, which shows that $y \in \fh(B^\circ)$.
		
		\item The inclusion $\fh(B) \subseteq B_u$ follows as in \ref{itm:fh:open}.
		Moreover, since $p_u \leq p$, the closed unit ball $B_u$ of $p_u$ is also closed with respect to $p$.
		It follows that $\overline{\fh(B)}^{\,p} \subseteq B_u$.
		For the converse, let $y \in B_u$.
		Then we have $\frac{n}{n+1} y \in B_u^\circ$ for all $n \in \N$, so it follows from \ref{itm:fh:open} that $\frac{n}{n+1}y \in \fh(B^\circ) \subseteq \fh(B)$.
		Letting $n \to \infty$, we find that $y \in \overline{\fh(B)}^{\,p}$.
		\qedhere
	\end{enumerate}
\end{proof}

\noindent
To illustrate the preceding result, we once again revisit \autoref{xmpl:R2-lq}.

\begin{example}[{cf.~\autoref{xmpl:R2-lq}}]
	Let $(E,E_+,p) = (\R^2,\R_{\geq 0}^2,\lVert \:\cdot\: \rVert_{\ell_q})$, where $\lVert \:\cdot\: \rVert_{\ell_q}$ denotes the $\ell_q$\nobreakdash-norm for $1 \leq q \leq \infty$, and let $p_u$ be the $1$\nobreakdash-max-normalization of $p$.
	Then the open unit ball $B_u^\circ$ of $p_u$ is depicted in \autoref{fig:R2-full-hull}.
	Note that $B_u^\circ$ can be thought of as the full hull of $B^\circ$, but also as $(\mathcal B \cap -\mathcal B)^\circ$, where $\mathcal B$ is the set drawn in \autoref{fig:R2-example}.
	In this example, it is not hard to see that $B_u^\circ = B^\circ$ if and only if $q = \infty$ (see also \myautoref{cor:seminorm-normalization}{itm:pu:equal} below).
	Moreover, when $q < \infty$, we have $p_u(1,-1) = 1$ whereas $p_u(1,1) = 2^{1/q} \neq 1$, which shows that the $1$\nobreakdash-max-normalization of a lattice norm is not necessarily a lattice norm.
\end{example}

\begin{figure}[h!t]
	\setlength{\belowcaptionskip}{-4pt}
	\centering
	\begin{tikzpicture}[scale=.65,
		            cone_style/.style={blue!60,fill opacity=0.5},
		            cone_label/.style={blue!30!black},
		            cone_dash/.style={blue!30!black,densely dotted,thin}]
		\begin{scope}[xshift=-6cm]
			\draw[gray!50,thin] (-2,-2) grid (2,2);
			\fill[cone_style] (-1,0) -- (0,-1) -| (1,0) -- (0,1) -| (-1,0);
			\draw[cone_dash] (0,-1) -- (1,0) (0,1) -- (-1,0);
			\draw[thick] (-2,0) -- (2,0) (0,-2) -- (0,2);
			\node[cone_label] at (-.65,.65) {$B_u^\circ$};
			\node[anchor=base] at (0,-3) {$q = 1$};
		\end{scope}
		\begin{scope}[xshift=0cm]
			\draw[gray!50,thin] (-2,-2) grid (2,2);
			\fill[cone_style] (-1,0) arc (180:270:1cm) -| (1,0) arc (0:90:1cm) -| (-1,0);
			\draw[cone_dash] (0,-1) arc (-90:0:1cm) (0,1) arc (90:180:1cm);
			\draw[thick] (-2,0) -- (2,0) (0,-2) -- (0,2);
			\node[cone_label] at (-.65,.65) {$B_u^\circ$};
			\node[anchor=base] at (0,-3) {$q = 2$};
		\end{scope}
		\begin{scope}[xshift=6cm]
			\draw[gray!50,thin] (-2,-2) grid (2,2);
			\fill[cone_style] (-1,-1) rectangle (1,1);
			\draw[thick] (-2,0) -- (2,0) (0,-2) -- (0,2);
			\node[cone_label] at (-.65,.65) {$B_u^\circ$};
			\node[anchor=base] at (0,-3) {$q = \infty$};
		\end{scope}
	\end{tikzpicture}
	\caption{The unit ball of the $1$\nobreakdash-max-normalization of $(\R^2,\R_{\geq 0}^2,\lVert \:\cdot\: \rVert_{\ell_q})$, for different values of $q$.}
	\label{fig:R2-full-hull}
\end{figure}

\begin{remark}
	The closure in \myautoref{prop:full-hull}{itm:fh:closed} can not always be omitted, as we now explain.
	
	A set $M \subseteq E$ is called \emph{proximinal} if the infimum $d(x,M) = \inf_{y \in M} p(x - y)$ is attained for all $x \in E$.
	Many standard examples of convex cones in normed spaces are proximinal, owing to the following three sufficient conditions for proximality:
	\begin{itemize}
		\item It is well-known that every non-empty weak\nobreakdash-$*$ closed set in the dual of a normed vector space is proximinal (see e.g.{} \cite[p. 239]{Phelps-approximation}), so the natural cone in the dual of a normed ordered vector space is proximinal.
		In particular, every closed convex cone in a reflexive Banach space is proximinal.
		
		\item It is not hard to see that the positive cone of every Banach lattice $E$ is proximinal (for every $x \in E$, its positive part $x^+$ is a closest vector in $E_+$).
		
		\item By a similar argument, for every $C^*$\nobreakdash-algebra $\mathcal A$, the positive cone of $\mathcal A$ is proximinal in the self-adjoint part of $\mathcal A$ (see also \autoref{prop:C*-algebra} below).
	\end{itemize}
	For a recent survey on proximinal sets, see \cite{Assadi-Haghshenas-Narang}.
	
	If $E_+$ is proximinal, then a straightforward modification of the proof of \myautoref{prop:full-hull}{itm:fh:open} shows that the closed unit ball of $p_u$ is the full hull of the closed unit ball of $p$ (without taking the closure).
	By the preceding discussion, this is the case for most standard examples of normed ordered vector spaces.
	Nevertheless, the following example shows that this is not always the case and that the closure in \myautoref{prop:full-hull}{itm:fh:closed} can not always be omitted.
\end{remark}

\begin{example}
	\label{xmpl:proximinal}
	Let $E = \ell^1$ with the usual norm $p = \lVert \:\cdot\: \rVert_1$ and the (non-standard) closed, proper, and generating cone $E_+ = \{ x \in \ell^1 \, : \, x \geq 0\ \text{pointwise and}\ \langle x, \varphi \rangle \geq 0\}$, where $\varphi = (\frac{1}{2},-\frac{1}{2},\frac{2}{3},-\frac{2}{3},\frac{3}{4},-\frac{3}{4},\ldots) \in (\ell^1)\topdual = \ell^\infty$.
	Furthermore, let $y = (-1,1,0,0,0,\ldots) \in \ell^1$.
	Then it not hard to see that $d(y,E_+) = \frac{3}{2}$ but $y$ has no closest point in $E_+$ (i.e.{} the infimum is not attained), and that $d(-y,E_+) = 1$.
	Therefore, $p_u(y) = \max(d(y,E_+),d(-y,E_+)) = \frac{3}{2}$.
	This means that $\frac{2}{3}y$ lies in the closed unit ball of $p_u$ but not in the full hull of the closed unit ball of $p$.
	Indeed, if $x \leq \frac{2}{3}y \leq z$ with $p(x),p(z) \leq 1$, then $\frac{2}{3}y - x \in E_+$ with $p(\frac{2}{3}y - (\frac{2}{3}y - x)) = p(x) \leq 1$, which is impossible because $d(\frac{2}{3}y,E_+) = 1$ but the infimum is not attained.
\end{example}

We proceed to prove the following consequences of \autoref{prop:full-hull}.

\begin{corollary}
	\label{cor:seminorm-normalization}
	Let $(E,E_+,p)$ be a seminormed preordered vector space, and let $p_u$ be the $1$\nobreakdash-max-normalization of $p$.
	Then:
	\begin{enumerate}[label=(\alph*)]
		\item\label{itm:pu:equal} $p = p_u$ if and only if $p$ is $1$\nobreakdash-max-normal;
		\item\label{itm:pu:largest} $p_u$ is the largest $1$\nobreakdash-max-normal seminorm smaller than $p$;
		\item\label{itm:pu:puu} $(p_u)_u = p_u$.
	\end{enumerate}
\end{corollary}
\begin{proof}
	\begin{enumerate}[label=(\alph*)]
		\item Let $B^\circ$ and $B_u^\circ$ denote the open unit balls of $p$ and $p_u$, respectively.
		If $p = p_u$, then $p$ is $1$\nobreakdash-max-normal.
		Conversely, if $p$ is $1$\nobreakdash-max-normal, then $B^\circ$ is already full, so it follows from \myautoref{prop:full-hull}{itm:fh:open} that $B^\circ = B_u^\circ$, hence $p = p_u$.
		
		\item For every seminorm $q \leq p$, we have $d_q(x,E_+) \leq d_p(x,E_+)$, and therefore $q_u \leq p_u$.
		Hence, if $q$ is a $1$\nobreakdash-max-normal seminorm smaller than $p$, then $q = q_u \leq p_u$, by \ref{itm:pu:equal}.
		
		\item Immediate from \ref{itm:pu:equal}, since $p_u$ is $1$\nobreakdash-max-normal.
		\qedhere
	\end{enumerate}
\end{proof}

\begin{corollary}
	Let $(E,E_+,p)$ be a seminormed preordered vector space.
	Then the canonical map $\phi : E \to \unit{E}$ is isometric if and only if $p$ is $1$\nobreakdash-max-normal.
\end{corollary}

Note that an isometric map from a seminormed space to a normed space is not necessarily injective.
The canonical map $\phi : E \to \unit{E}$ is injective if and only if $\overline{E_+}$ is a proper cone.

\myautoref{cor:seminorm-normalization}{itm:pu:puu} shows that the operation that maps a seminorm to its $1$\nobreakdash-max-normalization is idempotent.
In fact, more is true: $(E,E_+,p)$ and $(E,\overline{E_+},p_u)$ have the same unitization.

\begin{proposition}
	\label{prop:order-unitization-is-stable}
	Let $(E,E_+,p)$ be a seminormed preordered vector space, and let $p_u$ be the $1$\nobreakdash-max-normalization of $p$.
	Then $(E,E_+,p)$, $(E,\overline{E_+},p)$, $(E,E_+,p_u)$ and $(E,\overline{E_+},p_u)$ all have the same Archimedean order unitization.
\end{proposition}
\begin{proof}
	By \autoref{cor:extension-comparison}, $(E,E_+,p)$ and $(E,\overline{E_+},p_u)$ have the same continuous positive linear functionals of norm $\leq 1$, so it follows from \autoref{cor:alternative-formulas} that these two spaces have the same Archimedean order unitization.
	The other two sit in between, so they must also be equal.
\end{proof}

Next, we study under which conditions $p$ and $p_u$ are equivalent.
It is well known that every locally full seminorm $p$ is equivalent to a $1$\nobreakdash-max-normal seminorm (see for instance \cite[proof of Theorem 2.38]{Aliprantis-Tourky}).
Using the $1$\nobreakdash-max-normalization, we prove the following stronger result.

\begin{proposition}
	\label{prop:equivalent-normalization}
	Let $(E,E_+,p)$ be a seminormed preordered vector space, and let $p_u$ be the $1$\nobreakdash-max-normalization of $p$.
	Then for all $\alpha \geq 1$, the seminorm $p$ is $\alpha$\nobreakdash-max-normal if and only if $\frac{1}{\alpha} p \leq p_u \leq p$.
\end{proposition}
\begin{proof}
	If $\frac{1}{\alpha} p \leq p_u \leq p$, then for all $x,y,z \in E$ with $x \leq y \leq z$ we have
	\[ p(y) \leq \alpha p_u(y) \leq \alpha \max(p_u(x),p_u(z)) \leq \alpha \max(p(x),p(z)), \]
	because $p_u$ is $1$\nobreakdash-max-normal.
	Hence, it follows that $p$ is $\alpha$\nobreakdash-max-normal.
	
	For the converse, suppose that $p$ is $\alpha$\nobreakdash-max-normal.
	Let $y \in E$ be arbitrary, and choose sequences $\{x_n\}_{n=1}^\infty$ and $\{z_n\}_{n=1}^\infty$ in $E_+$ such that
	\[ d(y,E_+) = \lim_{n\to\infty} p(y - x_n), \qquad \text{and} \qquad d(-y,E_+) = \lim_{n\to\infty} p(-y - z_n). \]
	For all $n \in \N$ we have $y - x_n \leq y \leq y + z_n$, and therefore $p(y) \leq \alpha \max(p(y - x_n) , p(y + z_n))$, since $p$ is $\alpha$\nobreakdash-max-normal.
	It follows that
	\[ p_u(y) = \max(d(y,E_+),d(-y,E_+)) = \lim_{n\to\infty} \max(p(y - x_n) , p(-y - z_n)) \geq \tfrac{1}{\alpha} p(y). \]
	The inequality $p_u \leq p$ holds always, so we have $\frac{1}{\alpha} p \leq p_u \leq p$.
\end{proof}

In particular, this shows that every $\alpha$\nobreakdash-max-normal seminorm is $\alpha$\nobreakdash-equivalent to a $1$\nobreakdash-max-normal seminorm.
To our knowledge, this simple result is not particularly well known.

Finally, as an application, we turn to locally convex spaces with a normal cone.
Given a family of seminorms that generates the topology, we give a direct formula for a family of \emph{monotone} seminorms that also generates the topology.

\begin{theorem}[{$=$ \autoref{thm:intro:lcs-normal}}]
	\label{thm:lcs-normal}
	Let $(E,E_+,\lintop)$ be a locally convex ordered vector space whose positive cone is normal.
	Then for every family of seminorms $\{p_\lambda\}_{\lambda \in \Lambda}$ generating the topology of $E$, the corresponding family $\{p_{\lambda,u}\}_{\lambda \in \Lambda}$ of $1$\nobreakdash-max-normalizations also generates the topology of $E$.
	\textup(Here $p_{\lambda,u}(x) = \max(d_\lambda(x,E_+),d_\lambda(-x,E_+))$, where $d_\lambda(x,E_+) = \inf_{y\in E_+} p_\lambda(x - y)$.\textup)
\end{theorem}
\begin{proof}
	For $\lambda \in \Lambda$, let $B_\lambda^\circ$ and $B_{\lambda,u}^\circ$ denote the open unit balls of $p_\lambda$ and $p_{\lambda,u}$, respectively.
	
	Since $p_{\lambda,u} \leq p_\lambda$, it is clear that the topology $\lintop_u$ generated by $\{p_{\lambda,u}\}_{\lambda \in \Lambda}$ is coarser than $\lintop$.
	For the converse, let $\{V_\alpha\}_{\alpha \in A}$ be a neighbourhood base of $0$ consisting of full sets.
	For every $\alpha \in A$ there are $\lambda \in \Lambda$ and $\varepsilon > 0$ such that $\varepsilon B_\lambda^\circ \subseteq V_\alpha$, and therefore $\varepsilon B_{\lambda,u}^\circ = \fh(\varepsilon B_\lambda^\circ) \subseteq V_\alpha$ (since $V_\alpha$ is full).
	This shows that $\lintop_u$ is finer than $\lintop$.
\end{proof}

In particular, we recover the following well-known characterization of locally full seminorms, where condition \ref{itm:equiv:p_u} is new.

\begin{corollary}[{cf.~\cite[Proposition II.1.7]{Peressini}, \cite[Theorem 2.38]{Aliprantis-Tourky}}]
	\label{cor:equivalent-normalization}
	Let $(E,E_+,p)$ be a seminormed preordered vector space, and let $p_u$ be the $1$\nobreakdash-max-normalization of $p$.
	Then the following are equivalent:
	\begin{enumerate}[label=(\roman*)]
		\item\label{itm:equiv:p_u} $p$ and $p_u$ are equivalent seminorms;
		\item\label{itm:equiv:1-max-normal} $p$ is equivalent to a $1$\nobreakdash-max-normal seminorm;
		\item\label{itm:equiv:locally_full} $p$ is locally full;
		\item\label{itm:equiv:a-max-normal} $p$ is $\alpha$\nobreakdash-max-normal for some $\alpha \geq 1$.
	\end{enumerate}
\end{corollary}
\begin{proof}
	$\myref{itm:equiv:p_u} \Longrightarrow \myref{itm:equiv:1-max-normal}$. Trivial.
	
	$\myref{itm:equiv:1-max-normal} \Longrightarrow \myref{itm:equiv:locally_full}$. If $p$ is equivalent to a $1$\nobreakdash-max-normal seminorm $p'$, then the open balls of $p'$ form a neighbourhood base of $0$ consisting of full sets, so $p$ is locally full.
	
	$\myref{itm:equiv:locally_full} \Longrightarrow \myref{itm:equiv:p_u}$. If $p$ is locally full, then it follows from the proof\hair\footnote{For simplicity, we have stated \autoref{thm:lcs-normal} only for locally convex spaces, where the topology is assumed to be Hausdorff, so technically it does not apply to a seminormed space. However, the proof does not require this.} of \autoref{thm:lcs-normal} that $p$ and $p_u$ are equivalent.
	
	$\myref{itm:equiv:p_u} \Longleftrightarrow \myref{itm:equiv:a-max-normal}$. This follows from \autoref{prop:equivalent-normalization}.
\end{proof}

\section{Two worked examples: order unit spaces and \texorpdfstring{$C^*$}{C*}-algebras}
\label{sec:special-cases}

In this section, we show that the Archimedean order unitization $\unit{E}$ of an order unit space $E$ is just the direct sum of the Archimedeanization $\unitArch{E}$ with $\R$, and we show that the Archimedean order unitization of the self-adjoint part of a $C^*$\nobreakdash-algebra is order isomorphic with the self-adjoint part of the $C^*$\nobreakdash-algebra unitization.

\subsection{Order unit spaces}

If $(E,E_+,u)$ is an order unit space, then we have two ways of turning it into an AOU space: by the Archimedeanization of Paulsen and Tomforde (\autoref{thm:intro:Paulsen-Tomforde}), and by the Archimedean order unitization (\autoref{thm:intro:universal-property}).
We show that these two constructions give closely related results.

Recall from \cite[\S{}2.3]{Paulsen-Tomforde} that $\unitArch{E} = (E / N, \overline{E_+} + N , u + N)$, where $N = \overline{E_+} \cap -\overline{E_+}$ and $\overline{E_+} + N = \{x + N : x \in \overline{E_+}\}$.
If $(E,E_+,u)$ and $(F,F_+,w)$ are order unit spaces, then we understand $E \oplus F$ to be the order unit space $(E \oplus F , E_+ \oplus F_+ , u \oplus w)$.
We show that $\unit{E} \cong \unitArch{E} \oplus \R$ in this sense.
For this we need the following lemma.

\begin{lemma}
	\label{lem:order-unit-distance}
	Let $(E,E_+,u)$ be an order unit space equipped with the order unit seminorm.
	Then for all $x \in E$ and all $\lambda \in \R$, we have $d(x,E_+) \leq \lambda$ if and only if $\lambda \geq 0$ and $x + \lambda u \in \overline{E_+}$.
\end{lemma}
\begin{proof}
	Let $p$ denote the order unit seminorm.
	If $\lambda \geq 0$ and $x + \lambda u \in \overline{E_+}$, choose a sequence $\{y_n\}_{n=1}^\infty$ in $E_+$ such that $p(x + \lambda u - y_n) \leq \frac{1}{n}$ for all $n \in \N_1$.
	Then, $p(x - y_n) \leq p(\lambda u) + p(x + \lambda u - y_n) \leq \lambda + \frac{1}{n}$ for all $n \in \N_1$, which shows that $d(x,E_+)  \leq \lambda$.
	
	Conversely, if $d(x,E_+) \leq \lambda$, then it is immediate that $\lambda \geq 0$.
	Choose a sequence $\{z_n\}_{n=1}^\infty$ in $E_+$ such that $p(x - z_n) \leq \lambda + \frac{1}{n}$ for all $n \in \N_1$.
	Then we have $-(\lambda + \frac{2}{n})u \leq x - z_n \leq (\lambda + \frac{2}{n})u$, hence $0 \leq z_n \leq x + (\lambda + \frac{2}{n})u$.
	Since we have $p(x + \lambda u - (x + (\lambda + \frac{2}{n})u)) = p(-\frac{2}{n}u) \leq \frac{2}{n}$ for all $n \in \N_1$, it follows that $x + \lambda u \in \overline{E_+}$.
\end{proof}

\begin{corollary}
	If $(E,E_+,u)$ is an order unit space, then the map $\psi : \unit{E} \to \unitArch{E} \oplus \R$ given by $\psi(x + N,\lambda) = (x + \lambda u + N , \lambda)$ defines a unital order isomorphism.
\end{corollary}
\begin{proof}
	Since $\unit{E} = (E / N) \oplus \R = \unitArch{E} \oplus \R$, it is clear that $\psi$ is a well-defined linear isomorphism.
	By \autoref{def:minimal-unitization}, for all $x \in E$ and $\lambda \in \R$, we have $(x + N,\lambda) \in (\unit{E})_+$ if and only if $d(x,E_+) \leq \lambda$.
	On the other hand, we have $(x + \lambda u + N , \lambda) \in (\unitArch{E})_+ \oplus \R_{\geq 0}$ if and only if $x + \lambda u \in \overline{E_+}$ and $\lambda \geq 0$, which by \autoref{lem:order-unit-distance} is equivalent to $d(x,E_+) \leq \lambda$.
	This shows that $\psi$ is an order isomorphism.
	Moreover, since $\psi(0,1) = (u + N,1)$, we see that $\psi$ is unital.
\end{proof}

\subsection{The self-adjoint part of a \texorpdfstring{$C^*$}{C*}-algebra}

Let $\mathcal A$ be a $C^*$\nobreakdash-algebra, and let $\mathcal A^\sa$ be its self-adjoint part.
Equipped with the cone $\mathcal A_+$ of positive elements and the $C^*$\nobreakdash-algebra norm, $(\mathcal A^\sa , \mathcal A_+ , \lVert \:\cdot\: \rVert)$ becomes a normed ordered vector space.
Moreover, if $\mathcal A$ is unital, then $(\mathcal A^\sa, \mathcal A_+, 1)$ is an AOU space and the order unit norm on $\mathcal A^\sa$ agrees with the $C^*$\nobreakdash-algebra norm.

The $C^*$\nobreakdash-algebra unitization $\tilde{\mathcal A}$ of $\mathcal A$ is the $*$\nobreakdash-algebra on the vector space $\tilde{\mathcal A} = \mathcal A \oplus \C$ with multiplication given by $(a,\lambda)(b,\mu) = (ab + \lambda b + \mu a , \lambda\mu)$ and involution given by $(a,\lambda)^* = (a^*,\overline{\lambda})$.
This is a unital $*$\nobreakdash-algebra with unit $(0,1)$, it contains $\mathcal A$ as a subalgebra, and it admits a norm making it into a $C^*$\nobreakdash-algebra and extending the norm of $\mathcal A$ (see e.g.{} \cite[Theorem 2.1.6]{Murphy}).
The following result shows that the Archimedean order unitization of $\mathcal A^\sa$ coincides with the self-adjoint part of the $C^*$\nobreakdash-algebra unitization $\tilde{\mathcal A}$  (as AOU spaces).

\begin{proposition}
	\label{prop:C*-algebra}
	Let $\mathcal A$ be a $C^*$\nobreakdash-algebra, and let $\psi : \mathcal A^\sa \hookrightarrow \tilde{\mathcal A}^\sa$ be the inclusion of $\mathcal A^\sa$ in the $C^*$\nobreakdash-algebra unitization $\tilde{\mathcal A}$.
	Then the canonical map $\mUnit{\psi} : \unit{(\mathcal A^\sa)} \to \tilde{\mathcal A}^\sa$, $(a,\lambda) \mapsto (a,\lambda)$ is a unital order isomorphism.
\end{proposition}
\begin{proof}
	Since $\mathcal A_+$ is a closed proper cone, we have $N = \mathcal A_+ \cap -\mathcal A_+ = \{ 0 \}$, so $\mUnit{\psi}$ is a linear isomorphism.
	Clearly $\mUnit{\psi}$ is unital.
	We show that it is bipositive.
	
	By \autoref{lem:order-unit-distance}, for all $a \in \tilde{\mathcal A}^\sa$ and $\lambda \in \R$, we have $d(a,\tilde{\mathcal A}_+) \leq \lambda$ if and only if $\lambda \geq 0$ and $a + \lambda 1 \in \tilde{\mathcal A}_+$.
	By taking the Gelfand representation of the $C^*$\nobreakdash-subalgebra of $\tilde{\mathcal A}$ generated by $1$ and $a$, it is easy to see that the latter condition is equivalent to $\lambda \geq \lVert a^- \rVert$.
	It follows that $d(a,\tilde{\mathcal A}_+) = \lVert a^- \rVert$.
	
	Now let $a \in \mathcal A^\sa$.
	Then we have $d(a,\mathcal A_+) \geq d(\psi(a),\tilde{\mathcal A}_+) = \lVert \psi(a)^- \rVert = \lVert a^- \rVert$.
	Since $a^+ , a^- \in \mathcal A$ and $\lVert a - a^+ \rVert = \lVert a^- \rVert$, we have equality: $d(a,\mathcal A_+) = d(\psi(a), \tilde{\mathcal A}_+) = \lVert a^- \rVert$.
	
	Note that $(a,\lambda) \in \tilde{\mathcal A}_+$ implies that $\lambda \geq 0$, since the map $\tilde{\mathcal A} \to \C$, $(a,\lambda) \to \lambda$ is a $*$\nobreakdash-homomorphism and therefore positive.
	Therefore, we have $(a,\lambda) \in \tilde{\mathcal A}_+$ if and only if $\lambda \geq 0$ and $\psi(a) + \lambda 1 \in \tilde{\mathcal A}_+$, which by the above is equivalent to $d(a,\mathcal A_+) = d(\psi(a) , \tilde{\mathcal A}_+) \leq \lambda$.
	This shows that $\mUnit{\psi}$ is an order isomorphism.
\end{proof}

\section{Unitizations and representations}
\label{sec:representations}

In this section, we briefly discuss the role of the Archimedean order unitization in representation theory of ordered (topological) vector spaces, and we give an alternative construction of $\unit{E}$ in terms of representation theory.

By Kadison's representation theorem \cite{Kadison-representation}, every AOU space $(E,E_+,u)$ admits an isometric order embedding $\chi : E \to C(\Omega)$ for some compact Hausdorff space $\Omega$.
Hence, for every seminormed preordered vector space $(E,E_+,p)$, the composition of the canonical map $\phi : E \to \unit{E}$ and the Kadison representation $\chi : \unit{E} \to C(\Omega)$ yields a representation $E \to C(\Omega)$.
We show that this representation can also be obtained directly from $E$ (without passing through $\unit{E}$), using the following alternative construction of $\unit{E}$.

\begin{construction}
	\label{cstr:rep}
	Let $(E,E_+,p)$ be a seminormed preordered vector space.
	Then the set $\Omega := \{ \varphi \in E_+\topdual \, : \, \lVert \varphi \rVert \leq 1\}$ is weak\nobreakdash-$*$ compact, by the Banach--Alaoglu theorem.
	Let $\Omega_\weakstar$ denote the compact Hausdorff space thus obtained.
	For every $x \in E$, the evaluation function $\psi_x : \Omega_\weakstar \to \R$, $\varphi \mapsto \varphi(x)$ is continuous, so we have a well-defined linear function $\psi : E \to C(\Omega_\weakstar)$, $x \mapsto \psi_x$.
	Moreover, we have $\ker(\psi) = \bigcap_{\varphi \in E_+\topdual} \ker(\varphi) = \overline{E_+} \cap -\overline{E_+} = N$, so $\psi$ descends to a well-defined linear map $E/N \to C(\Omega_\weakstar)$, $x + N \mapsto \psi_x$.
	Define $\tilde\psi : \tilde E \to C(\Omega_\weakstar)$ by $\tilde\psi(x + N , \lambda) = \psi_x + \lambda 1$.
\end{construction}

\begin{proposition}
	\label{prop:C-Omega-construction}
	Let $\Omega_\weakstar$, $\psi$ and $\tilde\psi$ be as in \autoref{cstr:rep}.
	Then $\im(\psi) \cap \spn(1) = \{0\}$ and $\tilde\psi$ provides a unital order isomorphism $\tilde E \cong \im(\psi) \oplus \spn(1)$.
\end{proposition}
\begin{proof}
	Since $\Omega$ contains the zero functional $0 \in E_+\topdual$, for all $x \in E$ we have $\psi(x)(0) = 0$, which shows that $\im(\psi) \cap \spn(1) = \{0\}$.
	It follows that $\tilde\psi$ is injective and that $\im(\tilde\psi)$ is equal to the internal direct sum $\im(\psi) \oplus \spn(1)$.
	Clearly $\tilde\psi$ is unital, and it follows from \autoref{cor:alternative-formulas} that $\tilde\psi$ is bipositive, so we conclude that $\tilde\psi$ is a unital order isomorphism.
\end{proof}

It follows that the Archimedean order unitization $\unit{E}$ can also be constructed as $\im(\psi) \oplus \spn(1)$.
Moreover, since the state space of $\unit{E}$ is $\{(\varphi,1) \in \spn(E_+\topdual) \oplus \R \, : \, \varphi \in E_+\topdual, \lVert \varphi \rVert \leq 1\}$, it is not hard to see that the representation $\psi : E \to C(\Omega_\weakstar)$ from \autoref{cstr:rep} agrees with the composition of the canonical map $\phi : E \to \unit{E}$ and the Kadison representation $\chi : \unit{E} \to C(\Omega)$ on the full state space \cite[Lemma 2.5]{Kadison-representation}.%
	\hair\footnote{Kadison \cite[Theorem 2.1]{Kadison-representation} also gave a representation on $C(\overline{\ex(\Omega)}^{\,\weakstar})$, where $\overline{\ex(\Omega)}^{\,\weakstar}$ denotes the weak\nobreakdash-$*$ closure of the set of extreme points of $\Omega$. This is a restriction of the representation we considered above.}

We note that \autoref{prop:C-Omega-construction} also gives another proof of the following classical result.

\begin{corollary}[{\cite[Theorem~1']{Nachbin}, see also \cite[Theorem~4.2]{Riedl}}]
	\label{cor:isometric-representation}
	A normed ordered vector space $(E,E_+,\lVert \:\cdot\: \rVert)$ is isometrically order isomorphic to a subspace of $C(\Omega)$ for some compact Hausdorff space $\Omega$ if and only if $E_+$ is closed and $\lVert \:\cdot\: \rVert$ is $1$\nobreakdash-max-normal.
\end{corollary}
\begin{proof}
	If $E \subseteq C(\Omega)$, then $E$ is ordered by a closed cone and normed by a $1$\nobreakdash-max-normal norm.
	Conversely, suppose that $E_+$ is closed (and proper, because $E$ is ordered, not preordered) and that $\lVert \:\cdot\: \rVert$ is $1$\nobreakdash-max-normal.
	Then the canonical map $\phi : E \to \unit{E}$ is an isometric order embedding.
	Since \autoref{prop:C-Omega-construction} shows that $\unit{E} \cong \im(\psi) \oplus \spn(1) \subseteq C(\Omega_\weakstar)$, the result follows.
\end{proof}

\begin{remark}
	\label{rmk:representations}
	The preceding discussion shows that, in order to obtain a representation $E \to C(\Omega)$, it is not necessary to first adjoin an order unit and then apply the Kadison representation, because the same representation can already be obtained directly from $E$.
	In \cite{Dobben-semisimplicity}, we show that representations of arbitrary ordered topological vector spaces can be obtained in a similar way from collections of positive continuous linear functionals, without relying on any kind of normalization (with respect to some seminorm or order unit).
	This suggests that unitizations are not all that important in the representation theory of ordered vector spaces.
	This is quite surprising, given the similarity between Kadison's representation theorem and the Gelfand representation and the crucial role played by unitizations in the basic theory of operator algebras.
\end{remark}

\section{Applications to matrix ordered operator spaces}
\label{sec:operator-systems}

The Archimedean order unitization described in this paper is similar in spirit to the ``partial unitization'' of a matrix ordered operator space, introduced by Werner \cite{Werner} (see also \cite{Karn-unitization,Karn-corrigendum}, \cite[\S 3]{Blecher-et-al}).
In this section, we show that the Archimedean order unitization agrees with Werner's partial unitization at the ground level, and we use the results from the present paper to prove new results about Werner's partial unitization.

\subsection{Preliminaries on matrix ordered operator spaces}

We recall the basic definitions of operator spaces and matrix ordered $*$\nobreakdash-vector spaces, following Effros and Ruan \cite{Effros-Ruan-textbook} and Werner \cite{Werner}.

\begin{definition}[Matrix ordered operator spaces]
	\label{def:operator-spaces-systems}
	\leavevmode
	\begin{enumerate}[label=(\alph*)]
		\item For a complex vector space $E$, denote by $M_{m,n}(E)$ the space of $m \times n$ matrices with entries in $E$, and $M_n(E) = M_{n,n}(E)$.
		Given $x \in M_{p,q}(E)$ and scalar matrices $\alpha \in M_{m,p}(\C)$, $\beta \in M_{q,n}(\C)$, the product $\alpha x \beta$ is the matrix $\alpha x \beta = [ \sum_{k,l} \alpha_{ik} x_{kl} \beta_{lj} ] \in M_{m,n}(E)$.
		
		\item An \emph{abstract operator space} is a complex vector space $E$ together with a norm $\lVert \:\cdot\: \rVert_n$ on $M_n(E)$ for every $n \in \N_1$, such that
		\begin{enumerate}[label=(\textbf{M\arabic*}),leftmargin=*]
			\item\label{itm:operator-space-1} $\lVert x \oplus y \rVert_{m + n} = \max( \lVert x \rVert_m , \lVert y \rVert_n )$ for all $m,n \in \N_1$, $x \in M_m(E)$, $y \in M_n(E)$;
			\item\label{itm:operator-space-2} $\lVert \alpha x \beta \rVert_m \leq \lVert \alpha \rVert \lVert x \rVert_n \lVert \beta \rVert$ for all $m,n \in \N_1$, $\alpha \in M_{m,n}(\C)$, $x \in M_n(E)$, $\beta \in M_{n,m}(\C)$.
		\end{enumerate}
		
		\item Let $E,F$ be abstract operator spaces.
		A linear map $\varphi : E \to F$ naturally defines linear maps $\varphi_n : M_n(E) \to M_n(F)$ for all $n \in \N_1$, given by $[x_{ij}] \mapsto [\varphi(x_{ij})]$.
		The map $\varphi$ is \emph{completely bounded} if there is a constant $C > 0$ such that $\lVert \varphi_n(x) \rVert_n \leq C \lVert x \rVert_n$ for all $n \in \N_1$, $x \in M_n(E)$.
		A completely bounded linear map $\varphi : E \to F$ is called a \emph{complete isomorphism} if there is a constant $c > 0$ such that $\lVert \varphi_n(x) \rVert_n \geq c \lVert x \rVert_n$ for all $n \in \N_1$, $x \in M_n(E)$, and a \emph{complete isometry} if $\varphi_n$ is an isometry for all $n \in \N_1$.
		
		\item\label{itm:*-vs} A \emph{$*$\nobreakdash-vector space} is a complex vector space $E$ with a conjugate linear involution $x \mapsto x^*$.
		The set of self-adjoint elements of $E$ is denoted $E^\sa$.
		If $E$ is a $*$\nobreakdash-vector space and $n \in \N_1$, then we understand $M_n(E)$ to be a $*$\nobreakdash-vector space with involution $[x_{ij}]^* := [x_{ji}^*]$.
		
		For $*$\nobreakdash-vector spaces $E$ and $F$, there is a natural conjugate linear involution on the space $\Hom(E,F)$ of linear maps $E \to F$, given by $\overline{T}(x) := (T(x^*))^*$.
		We call this \emph{conjugation} and avoid the term ``adjoint'' in this context.%
			\hair\footnote{The adjoint of a linear map $T : E \to F$ is the map $T\topdual : F\topdual \to E\topdual$, $\varphi \mapsto \varphi \circ T$.
			Note furthermore that, according to the above definitions, the $*$\nobreakdash-operation on $M_n(\C)$ is $\alpha \mapsto \alpha^*$ (conjugate transpose), whereas the conjugation on $\Hom(\C^n,\C^n) \cong M_n(\C)$ is $\alpha \mapsto \overline{\alpha}$ (entry-wise conjugation).}
		A linear map $T : E \to F$ is self-conjugate if and only if $T[E^\sa] \subseteq F^\sa$.
		We write $\Hom(E,F)^\sc$ for the set of all self-conjugate linear maps.
		Every $\R$\nobreakdash-linear map $T : E^\sa \to F^\sa$ has a unique extension to a $\C$\nobreakdash-linear map $E \to F$, which is automatically self-conjugate, so we may identify $\Hom_\C(E,F)^\sc$ with $\Hom_\R(E^\sa , F^\sa)$.
		
		\item A \emph{matrix ordered vector space} is a $*$\nobreakdash-vector space $E$ with for every $n \in \N_1$ a convex cone $M_n(E)_+ \subseteq M_n(E)^\sa$ such that
		\begin{enumerate}[label=(\textbf{O\arabic*}),leftmargin=*]
			\item\label{itm:matrix-ordered-proper} $M_n(E)_+$ is proper for all $n \in \N_1$ (i.e.{} $M_n(E)_+ \cap -M_n(E)_+ = \{0\}$);
			\item\label{itm:matrix-ordered-2} $\alpha M_n(E)_+ \alpha^* \subseteq M_m(E)_+$ for all $m,n \in \N_1$, $\alpha \in M_{m,n}$.
		\end{enumerate}
		
		\item A \emph{matrix ordered AOU space} (or \emph{abstract unital operator system}) is a matrix ordered space $E$ with a designated element $u \in E$ (called the \emph{matrix order unit}) such that:
		\begin{enumerate}[label=(\textbf{UOS\arabic*}),leftmargin=*]
			\item $M_n(E)_+$ is Archimedean for all $n \in \N_1$;
			\item for all $n \in \N_1$, the diagonal matrix $u_n \in M_n(E)$ with all diagonal entries equal to $u$ is an order unit in $M_n(E)^\sa$.
		\end{enumerate}
		
		\item Let $E,F$ be matrix ordered spaces.
		A linear map $\varphi : E \to F$ is called \emph{completely positive} (resp.{} a \emph{complete order embedding}) if $\varphi_n$ is positive (resp.{} an order embedding) for all $n \in \N_1$.
		
		\item A \emph{matrix ordered operator space} is a space $E$ that is at the same time a matrix ordered vector space and an operator space.%
			\hair\footnote{Some authors additionally require that the involution is a complete isometry and the cones $M_n(E)_+$ are closed for all $n \in \N_1$, but we do not make these assumptions unless stated explicitly.}
		For the remainder of this paper, we will assume that all matrix ordered operator spaces are \emph{semisimple} (in the sense of \cite{Dobben-semisimplicity}), by which we mean that $\overline{M_n(E)_+} \cap -\overline{M_n(E)_+} = \{ 0 \}$ for all $n \in \N_1$ (it is sufficient to check this for $n = 1$).
		The results in this section also go through when $E$ is not semisimple, but then one has to quotient out the lineality space $N = \overline{M_n(E)_+} \cap -\overline{M_n(E)_+}$ of $\overline{M_n(E)_+}$, just like we did in the rest of this paper.
		We leave the details to the reader.
		
		\item The \emph{state space} of a matrix ordered operator space $E$ is the sequence $\{S_n(E)\}_{n=1}^\infty$ given by
		\begin{align*}
			S_n(E) &= \big\{ \varphi \in M_n(E)\topdual \, : \, \overline{\varphi} = \varphi , \ \varphi \geq 0 , \ \lVert \varphi \rVert = 1 \big\} \\
			&= \big\{ \varphi \in (M_n(E)^\sa)\topdual \, : \, \varphi \geq 0 , \ \lVert \varphi \rVert = 1 \big\},
		\end{align*}
		where we identify $(M_n(E)\topdual)^\sc$ with $(M_n(E)^\sa)\topdual$ as in \ref{itm:*-vs}.
	\end{enumerate}
\end{definition}

For a complex Hilbert space $\mathcal H$, the space $B(\mathcal H)$ becomes a matrix ordered operator space in a natural way, by identifying $M_n(B(\mathcal H)) \cong B(\mathcal H^n)$ and equipping this space with its $C^*$\nobreakdash-algebra norm and cone.
The Choi--Effros representation theorem \cite{Choi-Effros} shows that every matrix ordered AOU space is completely order isomorphic to a subspace of some $B(\mathcal H)$.
Similarly, Ruan's representation theorem \cite{Ruan} shows that every abstract operator space admits a complete isometry $E \to B(\mathcal H)$.

The Choi--Effros representation theorem was extended to non-unital matrix ordered operator spaces by Werner \cite[Theorem 4.15]{Werner}.
He showed that, if the cones $M_n(E)_+$ are closed and the $*$\nobreakdash-operation is a complete isometry, then $E$ admits a ``partial unitization'', which can then be used in conjunction with the Choi--Effros theorem to get a representation as a subspace of some $B(\mathcal H)$.
He proved that this representation is a complete isometry if and only if $\lVert \:\cdot\: \rVert$ agrees with the ``modified numerical radius'', but his definition of the latter was not very concrete yet.
Moreover, he announced a necessary and sufficient condition for this representation to be a complete isomorphism, but to our knowledge no proof exists in the literature.
Using the results from the present paper, we are able to address both of these issues.

Werner's construction goes as follows.

\begin{definition}[{Werner, \cite[Definition 4.7]{Werner}}]
	\label{def:partial-unitization}
	Let $E$ be a semisimple matrix ordered operator space.
	Then the \emph{partial unitization} of $E$ is the matrix ordered AOU space $\pu{E} = E \oplus \C$, where $(x,\alpha) \in M_n(E)^\sa \oplus M_n(\C)^\sa \cong M_n(\pu{E})^\sa$ is positive if and only if $\alpha \in M_n(\C)_+$ and
	\[  \varphi((\alpha + \varepsilon)^{-1/2} x (\alpha + \varepsilon)^{-1/2}) \geq -1 \qquad \text{for all $\varepsilon > 0$, $\varphi \in S_n(E)$} . \]
\end{definition}

Werner then showed that $\pu{E}$ is a matrix ordered AOU space that satisfies a universal property similar to that of \autoref{thm:intro:universal-property}, and that $E$ embeds in $\pu{E}$ via a split exact sequence
\[ \begin{tikzcd}[column sep=scriptsize]
	0 \arrow[r] & E \arrow[r, "\iota"] & \pu{E} \arrow[r,shift left=2pt, "\tau"] & \C \arrow[l,shift left=2pt, "\varepsilon"] \arrow[r] & 0,
\end{tikzcd} \]
where $\iota$ is a completely positive complete contraction and $\tau$ is unital.
In the presence of the universal property, this is just saying that $E$ embeds in $\pu{E}$ as a codimension $1$ subspace.%
	\hair\footnote{Let $\iota$ be the canonical map $E \to \unit{E}$, let $\tau$ be the extension of the zero map $E \to \C$, $x \mapsto 0$ to a unital completely positive linear map $\pu{E} \to \C$, and let $\varepsilon$ be the map $\lambda \mapsto \lambda u$, where $u$ is the unit of $\pu{E}$.}
This property was given more importance by Blecher, Kirkpatrick, Neal and Werner \cite{Blecher-et-al}, who showed that, apart from the ``smallest'' unitization given by Werner, under certain conditions there also exists a ``largest'' codimension $1$ unitization, which can be different.

\subsection{A simplification of Werner's partial unitization}
\label{subsec:pu-simplification}

Werner's condition of positivity in \autoref{def:partial-unitization} is a bit awkward to work with, for two reasons.
First, it relies on the state space, which suggests that one has to compute the state space before being able to tell which elements are positive in the unitization.
Second, multiplication on both sides by $(\alpha + \varepsilon)^{-1/2}$ for arbitrary $\varepsilon > 0$ makes the matrix computations rather messy.
We address both of these issues here by giving an equivalent geometric description of $M_n(\pu{E})_+$ analogous to our ``internal'' description of the Archimedean order unitization in \autoref{def:minimal-unitization}.

It follows from \autoref{cor:alternative-formulas} that Werner's partial unitization (\autoref{def:partial-unitization}) agrees at the ground level ($n = 1$) with our construction of the Archimedean order unitization (\autoref{def:minimal-unitization}).
Inspired by this, we proceed to use the results from this paper in order to show that Werner's partial unitization admits various equivalent geometric descriptions similar to \autoref{def:minimal-unitization}.
Some of these furthermore eliminate the need to multiply on both sides by matrices of the form $(\alpha + \varepsilon)^{\pm 1/2}$ for all $\varepsilon > 0$; see \myautoref{thm:pu-simplification}{itm:pu:alpha-ball} and \ref{itm:pu:inv}.

In what follows, let $d(S,T) = \inf_{s \in S , t \in T} \lVert s - t \rVert_n$ denote the distance between two sets $S,T \subseteq M_n(E)^\sa$, and let $B_n^\sa$ denote the closed unit ball of $M_n(E)^\sa$.

\begin{theorem}
	\label{thm:pu-simplification}
	Let $E$ be a semisimple matrix ordered operator space, and let $\pu{E}$ be its partial unitization.
	For all $n \in \N_1$ and all $(x,\alpha) \in M_n(E)^\sa \oplus M_n(\C)^\sa \cong M_n(\pu{E})^\sa$, the following are equivalent:
	\begin{enumerate}[label=(\roman*),series=Werner-simplification]
		\item\label{itm:pu:positive} $(x,\alpha) \in M_n(\pu{E})_+$;
		\item\label{itm:pu:alpha-epsilon-inv-dist} $\alpha \geq 0$ and $d\big( (\alpha + \varepsilon)^{-1/2} x (\alpha + \varepsilon)^{-1/2} \, , \, M_n(E)_+ \big) \leq 1$ for all $\varepsilon > 0$;
		\item\label{itm:pu:alpha-epsilon-ball} $\alpha \geq 0$ and $\big(x + (\alpha + \varepsilon)^{1/2} B_n^\sa (\alpha + \varepsilon)^{1/2} \big) \cap M_n(E)_+ \neq \varnothing$ for all $\varepsilon > 0$;
		\item\label{itm:pu:alpha-ball} $\alpha \geq 0$ and $d\big(x + \alpha^{1/2} B_n^\sa \alpha^{1/2} \, , \, M_n(E)_+ \big) = 0$.
	\end{enumerate}
	If $\alpha$ is invertible, then this is furthermore equivalent to:
	\begin{enumerate}[resume*=Werner-simplification]
		\item\label{itm:pu:inv} $\alpha \geq 0$ and $\inf \big\{ \big\lVert \alpha^{-1/2} (x - w) \alpha^{-1/2} \big\rVert_n \, : \, w \in M_n(E)_+ \big\} \leq 1$.
	\end{enumerate}
\end{theorem}
\pagebreak
\begin{proof}
	$\myref{itm:pu:positive} \Longleftrightarrow \myref{itm:pu:alpha-epsilon-inv-dist}$.
	\autoref{cor:alternative-formulas} shows that for all $y \in M_n(E)^\sa$, we have $\varphi(y) \geq -1$ for all $\varphi \in S_n(E)$ if and only if $d(y , M_n(E)_+) \leq 1$, so it follows immediately that $\myref{itm:pu:positive}$ and $\myref{itm:pu:alpha-epsilon-inv-dist}$ are equivalent.
	
	$\myref{itm:pu:alpha-epsilon-inv-dist} \Longrightarrow \myref{itm:pu:alpha-epsilon-ball}$.
	Assume that $\alpha \geq 0$ and $d\big( (\alpha + \varepsilon)^{-1/2} x (\alpha + \varepsilon)^{-1/2} \, , \, M_n(E)_+ \big) \leq 1$ for all $\varepsilon > 0$.
	Let $\varepsilon > 0$ be given, and choose some $\delta$ such that $0 < \delta < \varepsilon$.
	Then the eigenvalues of the self-adjoint matrix $(\alpha + \varepsilon)^{-1/2} (\alpha + \delta)^{1/2}$ are of the form $\sqrt{\frac{\lambda_i + \delta}{\lambda_i + \varepsilon}} < 1$ for some eigenvalue $\lambda_i \geq 0$ of $\alpha$, so we have $\eta := \lVert (\alpha + \varepsilon)^{-1/2} (\alpha + \delta)^{1/2} \rVert < 1$.
	Therefore $\frac{1}{\eta^2} > 1$, so by assumption we may choose $w \in M_n(E)_+$ and $y \in B_n^\sa$ such that
	\[  (\alpha + \delta)^{-1/2} x (\alpha + \delta)^{-1/2}  -  w  =  \tfrac{1}{\eta^2} \, y. \]
	Write $z := -\tfrac{1}{\eta^2} (\alpha + \varepsilon)^{-1/2} (\alpha + \delta)^{1/2} y (\alpha + \delta)^{1/2} (\alpha + \varepsilon)^{-1/2}$.
	Then we have
	\begin{gather*}
		\lVert z \rVert_n  \leq  \frac{1}{\eta^2} \big\lVert (\alpha + \varepsilon)^{-1/2} (\alpha + \delta)^{1/2} \big\rVert^2 \: \lVert y \rVert_n  =  \lVert y \rVert_n \leq 1 , \\
		(\alpha + \delta)^{1/2} w (\alpha + \delta)^{1/2}  =  x  -  \tfrac{1}{\eta^2} (\alpha + \delta)^{1/2} y (\alpha + \delta)^{1/2}  =  x  +  (\alpha + \varepsilon)^{1/2} z (\alpha + \varepsilon)^{1/2} ,
	\end{gather*}
	which shows that $\big(x + (\alpha + \varepsilon)^{1/2} B_n^\sa (\alpha + \varepsilon)^{1/2} \big) \cap M_n(E)_+ \neq \varnothing$.
	
	$\myref{itm:pu:alpha-epsilon-ball} \Longrightarrow \myref{itm:pu:alpha-ball}$.
	Assume that $\alpha \geq 0$ and $\big(x + (\alpha + \varepsilon)^{1/2} B_n^\sa (\alpha + \varepsilon)^{1/2} \big) \cap M_n(E)_+ \neq \varnothing$ for all $\varepsilon > 0$.
	Let $\delta > 0$ be given.
	Note that, by \ref{itm:operator-space-2} and the triangle inequality, we have
	\[ \lVert \beta z \beta - \gamma z \gamma \rVert_n   \leq   \lVert \beta \rVert \lVert z \rVert_n \lVert \beta - \gamma \rVert  +  \lVert \beta - \gamma \rVert \lVert z \rVert_n \lVert \gamma \rVert \qquad \text{($z \in M_n(E)$, $\beta,\gamma \in M_n(\C)$)}. \]
	Thus, since $\lim_{\varepsilon \to 0^+} (\alpha + \varepsilon)^{1/2} = \alpha^{1/2}$, we may choose some small enough $\varepsilon > 0$ such that
	\begin{align}
		\big\lVert (\alpha + \varepsilon)^{1/2} z (\alpha + \varepsilon)^{1/2}  -  \alpha^{1/2} z \alpha^{1/2} \big\rVert_n  \leq  \delta \lVert z \rVert_n \qquad \text{for all $z \in M_n(E)$}. \label{eqn:pu:epsilon-delta}
	\end{align}
	By assumption, we may choose $y \in B_n^\sa$ and $w \in M_n(E)_+$ such that $x + (\alpha + \varepsilon)^{1/2} y (\alpha + \varepsilon)^{1/2}  =  w$.
	Now it follows from \eqref{eqn:pu:epsilon-delta} that
	\begin{align*}
		\big\lVert x  +  \alpha^{1/2} y \alpha^{1/2}  -  w \big\rVert_n  &=  \big\lVert \alpha^{1/2} y \alpha^{1/2}  -  (\alpha + \varepsilon)^{1/2} y (\alpha + \varepsilon)^{1/2} \big\rVert_n  \leq  \delta \lVert y \rVert_n \leq \delta.
	\end{align*}
	This shows that $d(x + \alpha^{1/2} B_n^\sa \alpha^{1/2} , M_n(E)_+ ) \leq \delta$.
	Letting $\delta \to 0$ proves the result.
	
	$\myref{itm:pu:alpha-ball} \Longrightarrow \myref{itm:pu:alpha-epsilon-inv-dist}$.
	Assume that $\alpha \geq 0$ and $d\big(x + \alpha^{1/2} B_n^\sa \alpha^{1/2} \, , \, M_n(E)_+ \big) = 0$.
	Let $\varepsilon > 0$ be given, and write $\alpha_\varepsilon := \alpha + \varepsilon$.
	Since the eigenvalues of the self-adjoint matrix $\alpha_\varepsilon^{-1/2} \alpha^{1/2}$ are of the form $\frac{\sqrt{\lambda_i}}{\sqrt{\lambda_i + \varepsilon}} < 1$ for some eigenvalue $\lambda_i \geq 0$ of $\alpha$, we have $\lVert \alpha_\varepsilon^{-1/2} \alpha^{1/2} \rVert < 1$.
	Moreover, we have $\lVert \alpha_\varepsilon^{-1/2} \rVert \leq \frac{1}{\sqrt{\varepsilon}}$.
	
	By assumption, for any $\delta > 0$ we may choose $y \in M_n(E)^\sa$ with $\lVert y \rVert_n \leq 1$  and $w \in M_n(E)_+$ such that $\lVert x + \alpha^{1/2} y \alpha^{1/2}  -  w \rVert_n \leq \delta\varepsilon$.
	It follows that
	\begin{align*}
		\big\lVert \alpha_\varepsilon^{-1/2} (x - w) \alpha_\varepsilon^{-1/2} \big\rVert_n  &\leq  \big\lVert \alpha_\varepsilon^{-1/2} \alpha^{1/2} y \alpha^{1/2} \alpha_\varepsilon^{-1/2} \big\rVert_n  +  \big\lVert \alpha_\varepsilon^{-1/2} ( x + \alpha^{1/2} y \alpha^{1/2} - w ) \alpha_\varepsilon^{-1/2} \big\rVert_n \\
		&\leq  \big\lVert \alpha_\varepsilon^{-1/2} \alpha^{1/2} \big\rVert^2  \lVert y \rVert_n   +   \big\lVert \alpha_\varepsilon^{-1/2} \big\rVert^2  \big\lVert x + \alpha^{1/2} y \alpha^{1/2} - w \big\rVert_n \\
		&\leq  1  +  \delta.
	\end{align*}
	Since $\alpha_\varepsilon^{-1/2} w \alpha_\varepsilon^{-1/2} \in M_n(E)_+$, it follows that $d( \alpha_\varepsilon^{-1/2} x \alpha_\varepsilon^{-1/2}  ,  M_n(E)_+) \leq 1 + \delta$.
	Letting $\delta \to 0$ shows that $d( \alpha_\varepsilon^{-1/2} x \alpha_\varepsilon^{-1/2}  ,  M_n(E)_+) \leq 1$.
	
	$\myref{itm:pu:alpha-ball} \Longleftrightarrow \myref{itm:pu:inv}$, $\alpha$ invertible.
	Let $\delta > 0$.
	On the one hand, if $d(x + \alpha^{1/2} B_n^\sa \alpha^{1/2}  , M_n(E)_+ ) \leq \delta$, then we may choose $y \in B_n^\sa$ and $w \in M_n(E)_+$ such that $\lVert x + \alpha^{1/2} y \alpha^{1/2} - w \rVert \leq 2\delta$, hence
	\begin{align*}
		\big\lVert \alpha^{-1/2} (x - w) \alpha^{-1/2} \big\rVert_n  &\leq  \lVert y \rVert_n  +  \big\lVert \alpha^{-1/2} \big( x + \alpha^{1/2} y \alpha^{1/2} - w \big) \alpha^{-1/2} \big\rVert_n  \leq  1 + 2\delta \lVert \alpha^{-1} \rVert.
	\end{align*}
	On the other hand, if $\inf \big\{ \big\lVert \alpha^{-1/2} (x - w) \alpha^{-1/2} \big\rVert_n \, : \, w \in M_n(E)_+ \big\} \leq 1 + \delta$, then we may choose $y \in B_n^\sa$ and $w \in M_n(E)_+$ such that $\alpha^{-1/2} (x - w) \alpha^{-1/2} = -(1 + 2\delta) y$, hence
	\begin{align*}
		\big\lVert x + \alpha^{1/2} y \alpha^{1/2}  -  w \big\rVert_n  &=  2\delta \big\lVert \alpha^{1/2} y \alpha^{1/2} \big\rVert_n  \leq  2\delta \lVert \alpha \rVert.
	\end{align*}
	Letting $\delta \to 0$ proves that \myref{itm:pu:alpha-ball} and \myref{itm:pu:inv} are equivalent when $\alpha$ is invertible.
\end{proof}

\begin{remark}
	The formula in \myautoref{thm:pu-simplification}{itm:pu:alpha-epsilon-ball}, except with the closed unit ball replaced by the open unit ball, was also considered as the definition of a unitization by Karn in the proof of \cite[Theorem 2.8]{Karn-unitization}, but this was retracted and replaced by a more complicated formula in the corrigendum \cite[Theorem 0.1]{Karn-corrigendum}.
	It was shown in \cite[\S{}3]{Blecher-et-al} that Karn's unitization coincides with Werner's partial unitization, so the simpler formulas from \myautoref{thm:pu-simplification}{itm:pu:alpha-ball} and \ref{itm:pu:inv} also apply to Karn's unitization.
\end{remark}

\begin{remark}
	We point out that the criterion in \myautoref{thm:pu-simplification}{itm:pu:alpha-ball} does not imply that the sets $x + \alpha^{1/2} B_n^\sa \alpha^{1/2}$ and $M_n(E)_+$ have non-empty intersection, even if $M_n(E)_+$ is closed.
	This can already fail at the ground level ($n = 1$).
	For example, if $E_+$ is a closed convex cone that is not proximinal (see for instance \autoref{xmpl:proximinal}), then there exists a vector $x \in E^\sa$ such that $d(x,E_+) = 1$ but the infimum is not attained.
	Consequently, one has $d(x + B_1^\sa ,  E_+) = 0$ but the sets $x + B_1^\sa$ and $E_+$ are disjoint.
	
	If $E$ is finite\nobreakdash-dimensional, then $B_n^\sa$ is compact, so a straightforward compactness argument shows that the infimum $d(x + \alpha^{1/2} B_n^\sa \alpha^{1/2} , \overline{M_n(E)_+} )$ is attained.
	Since we have $d(S,T) = d(\overline{S},\overline{T})$, the following simplification of \autoref{thm:pu-simplification} in the finite\nobreakdash-dimensional case is immediate.
\end{remark}

\begin{corollary}
	\label{cor:pu-finite-dimensional}
	Let $E$ be a finite\nobreakdash-dimensional semisimple matrix ordered operator space, and let $\pu{E}$ be its partial unitization.
	Then for all $n \in \N_1$ and all $(x,\alpha) \in M_n(\pu{E})^\sa$, we have $(x,\alpha) \in M_n(\pu{E})_+$ if and only if
	$\alpha \geq 0$ and $\big(x + \alpha^{1/2} B_n^\sa \alpha^{1/2} \big) \cap \overline{M_n(E)_+} \neq \varnothing$.
\end{corollary}

The description of the positive cone given in \autoref{thm:pu-simplification} also gives us the following direct formula for the order unit norm on the self-adjoint elements, generalizing \autoref{prop:simple-formula} to matrix ordered operator spaces.

\begin{corollary}
	\label{cor:pu-norm-formula}
	Let $E$ be a semisimple matrix ordered operator space, and let $\pu{E}$ be its partial unitization.
	Then for all $n \in \N_1$, the order unit norm on $M_n(\pu{E})^\sa$ is given by
	\begin{align*}
		\lVert (x,\alpha) \rVert_{n,u}  &=  \inf \left\{ \lambda \geq \lVert \alpha \rVert \, \left| \ \begin{aligned}
			&\big(x + (\alpha + \lambda)^{1/2} B_n^\sa (\alpha + \lambda)^{1/2} \big) \cap M_n(E)_+ \neq \varnothing, \\
			&\big( {} {-x} + (-\alpha + \lambda)^{1/2} B_n^\sa (-\alpha + \lambda)^{1/2} \big) \cap M_n(E)_+ \neq \varnothing
		\end{aligned} \right.\right\} \\\noalign{\smallskip}
		&=  \inf \left\{ \lambda \geq \lVert \alpha \rVert \, \left| \, \exists y,z \in B_n^\sa \ \text{such that}\ \ \begin{aligned}
			&(\alpha + \lambda)^{1/2} y (\alpha + \lambda)^{1/2} \geq -x , \\
			&(-\alpha + \lambda)^{1/2}  z (-\alpha + \lambda)^{1/2} \geq x
		\end{aligned} \right. \right\} .
	\end{align*}
	Moreover, when $\alpha = \xi I_n$ for some $\xi \in \R$, one has
	\[ \lVert (x,\xi I_n) \rVert_{n,u}  =  \max\big( \hair d(x, M_n(E)_+) - \xi \, , \, d(-x, M_n(E)_+) + \xi \hair \big). \]
\end{corollary}
\begin{proof}
	By \myautoref{thm:pu-simplification}{itm:pu:alpha-epsilon-ball}, for all $\lambda \in \R_{\geq 0}$ we have $(0,-\lambda I_n) \leq (x,\alpha) \leq (0,\lambda I_n)$ if and only if $-\lambda I_n \leq \alpha \leq \lambda I_n$ and
	\begin{alignat*}{2}
		&\big(   x  + ( \alpha + \lambda + \varepsilon)^{1/2} B_n^\sa ( \alpha + \lambda + \varepsilon)^{1/2} \big)  \cap  M_n(E)_+  \neq \varnothing  & \qquad & \text{for all $\varepsilon > 0$}; \\
		&\big( {-x} + (-\alpha + \lambda + \varepsilon)^{1/2} B_n^\sa (-\alpha + \lambda + \varepsilon)^{1/2} \big)  \cap  M_n(E)_+  \neq \varnothing  & \qquad & \text{for all $\varepsilon > 0$}.
	\end{alignat*}
	Since we have $-\lambda I_n \leq \alpha \leq \lambda I_n$ if and only if $\lambda \geq \lVert \alpha \rVert$, and since possibly removing the smallest element of an upwards closed set does not affect the infimum, the first formula for $\lVert (x,\alpha) \rVert_{n,u}$ follows.
	Rewriting these conditions on $\lambda$, the second formula follows immediately.
	
	To prove the formula for $\lVert (x , \xi I_n) \rVert_{n,u}$, note that by \myautoref{thm:pu-simplification}{itm:pu:inv} for all $\lambda \in \R$ we have $(x,\lambda I_n) \in M_n(\pu{E})_+$ if and only if $d(x, M_n(E)_+) \leq \lambda$.
	The formula now follows by the argument of \autoref{prop:simple-formula}.
\end{proof}

\subsection{Criteria for the embedding \texorpdfstring{$E \to \pu{E}$}{E -> E\textasciicircum{}pu} to be a complete isomorphism}
\label{subsec:non-unital-operator-systems}

In \cite[Remark 4.14(iii)]{Werner}, Werner announced the following result:

\begin{claim}[{\cite[Remark 4.14(iii)]{Werner}}]
	\label{claim:non-unital-operator-systems}
	Let $E$ be a matrix ordered operator space such that the $*$\nobreakdash-operation is a complete isometry and $M_n(E)_+$ is closed for all $n \in \N_1$.
	Then the embedding $\iota : E \to \pu{E}$ is a complete isomorphism \textup(that is, $\lVert \iota^{-1} \rVert_{cb} < \infty$\textup) if and only if there is a constant $\kappa > 0$ such that $\lVert \:\cdot\: \rVert_n$ is $\kappa$\nobreakdash-max-normal for all $n \in \N_1$.
\end{claim}

To our knowledge, no proof of this claim can be found in the literature.
In this section, we use the results from the present paper to prove this claim.
Furthermore, we provide the precise relation between the constants $\lVert \iota^{-1} \rVert_{cb}$ and $\kappa$, which was not stated in Werner's paper.
For this we borrow the following notation from Werner's paper.

\begin{definition}[{\cite[\S{}3]{Werner}}]
	\label{def:modified-numerical-radius}
	Let $E$ be a semisimple matrix ordered operator space, and let $x \in M_n(E)$.
	The \emph{numerical radius} of $x$ is defined as
	\[ \nu_E^0(x) := \max \left\{ |\varphi(x)| \: : \: \varphi \in S_n(E) \right\}. \]
	The \emph{modified numerical radius} of $x$ is defined as
	\[ \nu_E(x) := \nu_E^0 \begin{pmatrix} 0 & x \\ x^* & 0 \end{pmatrix}. \]
\end{definition}

It follows from \autoref{cor:alternative-formulas} that $\nu_E^0$ agrees on the self-adjoint elements with the $1$\nobreakdash-max-normalization $\lVert \:\cdot\: \rVert_{n,u}$ of $\lVert \:\cdot\: \rVert_n$.
Moreover, in \cite[Lemma 4.8(c)]{Werner}, Werner showed that $\nu_E$ is equal to the norm inherited from the partial unitization, which by the final formula in \autoref{cor:pu-norm-formula} (with $\xi = 0$) also agrees on the self-adjoint elements with $\lVert \:\cdot\: \rVert_{n,u}$.
In summary, we have
\begin{align}
	\nu_E(x) = \nu_E^0(x) = \lVert x \rVert_{n,u} = \max \big( d(x, M_n(E)_+) , d(-x, M_n(E)_+) \big) \qquad \text{($x \in M_n(E)^\sa$)}. \label{eqn:radius-equals-modified-radius}
\end{align}
Using this, we can prove the following precise version of \autoref{claim:non-unital-operator-systems}.

\begin{proposition}[{cf.~\cite[Remark~4.14(iii)]{Werner}}]
	\label{prop:non-unital-operator-systems}
	Let $E$ be a semisimple matrix ordered operator space such that the $*$\nobreakdash-operation is a complete isometry.
	Then for all $\kappa \geq 1$, the following are equivalent:
	\begin{enumerate}[label=(\roman*)]
		\item For all $n \in \N_1$ and all $x \in M_n(E)$, one has $\nu_E(x) \geq \frac{1}{\kappa} \lVert x \rVert_n$;
		\item For all $n \in \N_1$, the norm $\lVert \:\cdot\: \rVert_n$ is $\kappa$\nobreakdash-max-normal on $M_n(E)^\sa$.
	\end{enumerate}
\end{proposition}
\begin{proof}
	By \autoref{prop:equivalent-normalization} and \myeqautoref{eqn:radius-equals-modified-radius}, the norm $\lVert \:\cdot\: \rVert_n$ is $\kappa$\nobreakdash-max-normal on $M_n(E)^\sa$ if and only if $\frac{1}{\kappa} \lVert x \rVert_n \leq \lVert x \rVert_{n,u} = \nu_E(x) = \nu_E^0(x)$ for all $x \in M_n(E)^\sa$.
	It remains to show that, if this holds for all $n \in \N_1$, then we also have $\nu_E(x) \geq \frac{1}{\kappa} \lVert x \rVert_n$ for all $n \in \N_1$ and all $x \in M_n(E)$ (not just the self-adjoint elements).
	To that end, note that for all $x \in M_n(E)$ we have
	\begin{align*}
		\lVert x \rVert_n = \max(\lVert x \rVert_n , \lVert x^* \rVert_n) = \left\lVert \begin{pmatrix} x & 0\\ 0 & x^* \end{pmatrix} \right\rVert_{2n} = \left\lVert \begin{pmatrix} 0 & x \\ x^* & 0 \end{pmatrix} \right\rVert_{2n},
	\end{align*}
	where the first equality follows because we assumed that the $*$\nobreakdash-operation is a complete isometry,
	the second equality follows from \ref{itm:operator-space-1},
	and the third equality follows because the norm is invariant under (left or right) multiplication by unitary scalar matrices.
	Thus, if $\lVert \:\cdot\: \rVert_{2n}$ is $\kappa$\nobreakdash-max normal on the self-adjoint elements, then we have
	\[ \frac{1}{\kappa} \lVert x \rVert_n = \frac{1}{\kappa} \left\lVert \begin{pmatrix} 0 & x \\ x^* & 0 \end{pmatrix} \right\rVert_{2n} \leq \nu_E^0 \begin{pmatrix} 0 & x \\ x^* & 0 \end{pmatrix} = \nu_E(x) \qquad \text{for all $x \in M_n(E)$}, \]
	since $\left( \begin{smallmatrix} 0 & x \\ x^* & 0 \end{smallmatrix} \right)$ is self-adjoint.
\end{proof}

Combining this with \cite[Theorem 4.15]{Werner}, we find the following characterization of ``isomorphic abstract non-unital operator systems''.

\begin{corollary}
	\label{cor:isomorphic-operator-system}
	Let $E$ be a matrix ordered operator space, and let $\kappa \geq 1$.
	Then there exists a Hilbert space $\mathcal H$, a self-adjoint subspace $F \subseteq B(\mathcal H)$ and a self-adjoint completely isomorphic complete order isomorphism $T : E \to F$ with $\lVert T \rVert_{cb} \leq 1$, $\lVert T^{-1} \rVert_{cb} \leq \kappa$ if and only if $E$ satisfies the following properties:
	\begin{enumerate}[label=(\roman*)]
		\item the cones $M_n(E)_+$ are closed;
		
		\item the $*$\nobreakdash-operation is a complete isometry;
		
		\item for all $n \in \N_1$, the norm $\lVert \:\cdot\: \rVert_n$ is $\kappa$\nobreakdash-max-normal on $M_n(E)^\sa$.
	\end{enumerate}
\end{corollary}

\pagebreak
In particular, for $\kappa = 1$, we get the following simple axiomatic characterization of ``abstract non-unital operator systems'', which is the operator space analogue of \autoref{cor:isometric-representation}.

\begin{corollary}[{cf.~\cite[Theorem~1.1]{Russell}}]
	\label{cor:isometric-operator-system}
	A matrix ordered operator space $E$ admits a completely isometric complete order isomorphism to a self-adjoint subspace of $B(\mathcal H)$ for some Hilbert space $\mathcal H$ if and only if $E$ satisfies the following properties:
	\begin{enumerate}[label=(\textbf{OS\arabic*}),leftmargin=*]
		\item the cones $M_n(E)_+$ are closed;
		
		\item the $*$\nobreakdash-operation is a complete isometry;
		
		\item for all $n \in \N_1$, the norm $\lVert \:\cdot\: \rVert_n$ is $1$\nobreakdash-max-normal on $M_n(E)^\sa$.
	\end{enumerate}
\end{corollary}

\begin{remark}
	In \mysecref{sec:representations}, we saw that a representation $E \to C(\Omega)$ of a normed ordered vector space $E$ can be obtained directly from $E$, without passing through the unitization (see \autoref{rmk:representations}).
	The same phenomenon can be observed for matrix ordered operator spaces.
	Indeed, Karn \cite[Theorem~1.8]{Karn-embedding} gave a direct representation from an abstract non-unital operator system to a $C^*$\nobreakdash-algebra, without passing through the partial unitization.
	Using the results from the present paper, it is easy to see that Karn's conditions in \cite[Theorem~1.8]{Karn-embedding} are equivalent to the conditions in \autoref{cor:isometric-operator-system}.
\end{remark}

\small

\bibliographystyle{my-alphaurl}
\bibliography{orderunits}

\end{document}